\documentclass[11pt,reqno]{amsart} 
\usepackage[a4paper,left=24mm,right=24mm,top=32mm,bottom=30mm]{geometry} 

\usepackage[latin1]{inputenc}
\usepackage[ngerman,english]{babel}

\usepackage[T1]{fontenc}
\usepackage{lmodern}
\usepackage{a4wide}
\usepackage{color}
\usepackage{enumerate}
\usepackage{amsmath}%
\usepackage{amsfonts}%
\usepackage{amssymb}
\usepackage{amsthm}
\usepackage{amsmath,amsxtra,amssymb,latexsym, amscd,amsthm}
\usepackage{bbm}
\usepackage{empheq}
\usepackage{graphicx}
\usepackage{MnSymbol}
\usepackage{wrapfig}
\usepackage[percent]{overpic}
\usepackage{setspace}
\usepackage{eurosym}
\usepackage{cite}
\usepackage{faktor}
\usepackage{xfrac}
\usepackage{esint}

\usepackage[section]{placeins}
\usepackage{enumitem}
\usepackage{xcolor}

\newtheorem{theorem}{Theorem}[section]
\newtheorem{proposition}[theorem]{Proposition}
\theoremstyle{definition}

\newtheorem{condition}[theorem]{Condition}

\newtheorem{definition}[theorem]{Definition}

\newtheorem{lemma}[theorem]{Lemma}

\newtheorem{remark}[theorem]{Remark}

\numberwithin{equation}{section}

\providecommand{\norm}[1]{\left\lVert #1 \right\rVert}
\providecommand{\abs}[1]{\left\lvert #1 \right\rvert}

\newcommand{\F}{\mathcal F}                
\newcommand{\eps}{\varepsilon}             
\newcommand{\pd}{\partial}                 
\renewcommand{\d}{{\rm{d}}}                
\newcommand{\IS}[1]{\mathcal{#1}}	   
\newcommand{\R}{\mathbb{R}}                
\newcommand{\N}{\mathbb{N}}                
\newcommand{\SL}{\Delta_{\Gamma}}          
\newcommand{\SG}{\nabla_{\Gamma}}          
\newcommand{\G}{\Gamma}
\newcommand{\C}[1]{{C^#1(\Gamma)}} 	   
\renewcommand{\L}[1]{{L^#1(\Gamma)}}       
\newcommand{\jump}[1]{\left[ #1 \right]_-^+}

\newcommand{\mres}{\mathbin{\vrule height 1.6ex depth 0pt width
		0.15ex\vrule height 0.15ex depth 0pt width 0.6ex}}

\renewcommand{\setminus}{\mathbin{\backslash}}

\newcommand{\thmcite}[2]{\hspace{-0.02em}{\cite[#1]{#2}}}

\newcommand{\supp}{\operatorname{supp}} 
 
\newcommand{\epsk}{{\varepsilon_k}}
\newcommand{\Id}{\operatorname{Id}}

\renewcommand{\div}{\operatorname{div}}

\newcommand{\solspace}{\mathcal{W}}

\newcommand{\X}{\mathcal{X}}
\newcommand{\M}{\mathcal{M}}

\newcommand{\dprodH}[3]{\left\langle #2, #3 \right\rangle_{H^{-1}(#1),H^1(#1)}}

\title[The Sharp Interface Limit of a Model for Phase Separation on Membranes]{On the Sharp Interface Limit of a Model for Phase Separation on Biological Membranes}
\author[H. Abels]{Helmut Abels}
\address{Helmut Abels, Fakult\"at f\"ur Mathematik, Universit\"at Regensburg, 93040 Regensburg, Germany, e-mail: helmut.abels@mathematik.uni-regensburg.de}
\author[J. Kampmann]{Johannes Kampmann}
\address{Johannes Kampmann, Fakult\"at f\"ur Mathematik, Universit\"at Regensburg, 93040 Regensburg, Germany, e-mail: johannes.kampmann@mathematik.uni-regensburg.de}

\begin{document}

\begin{abstract}
We rigorously prove the convergence of weak solutions to a model for lipid raft formation in cell membranes which was recently proposed by Garcke et al.\ \cite{GKRR} to weak (varifold) solutions of the corresponding sharp-interface problem for a suitable subsequence. In the system a Cahn-Hilliard type equation on the boundary of a domain is coupled to a diffusion equation inside the domain.
The proof builds on techniques developed by Chen \cite{ChenCH} for the corresponding result for the Cahn-Hilliard equation. 
\end{abstract}


\maketitle

\section{Introduction}

In \cite{GKRR} Garcke, Rätz, Röger and the second author proposed a model for phase separation on biological membranes based on the interplay between a thermodynamic equilibrium process and nonequilibrium effects, in particular active transport processes on the cell membrane.

Cell membranes consist of saturated and unsaturated lipid molecules which arrange themselves in a bilayer structure. Moreover, other molecules such as cholesterols or proteins are included. The lateral organisation of these different components is important for the functioning of the cell, contributing to protein trafficking, endocytosis, and signalling \cite{FSH2, RL}.  

A lot of attention in this context is given to the emergence of so-called lipid rafts. These rafts are intermediate sized domains ($10-200$  nm), characterized as regions consisting mainly of saturated lipid molecules enriched with cholesterols \cite{PIK}. We refer the reader to the overview \cite{FS} and the list of references therein for a discussion of the experimental evidence for their existence.
 
The model by Garcke et al.\ is a phase-field model derived from thermodynamic conservation laws, both on the membrane and in the cytosol. The former describes the phase separation between saturated and unsaturated lipid molecules, from which the lipid rafts emerge. The latter describes the dynamic inside the cytosol. The equations on the membrane and in the cytosol are then coupled by an in-/out-flux $q$ related to exchange processes between the cell and its membrane. From the viewpoint of thermodynamics, this exchange term can be interpreted as an external source term in both the membrane and cytosol equations. 

In order to introduce the model, let $B \subset \R^3$ be a bounded domain with smooth boundary $\G:=\partial B.$ The set $B$ and the surface $\G$ represent the cell and its outer membrane respectively. The basic quantities in the model are the rescaled relative concentration $\varphi$ of saturated lipids in the membrane, the relative concentration $v$ of membrane-bound cholesterol and the relative concentration $u$ of cytosolic cholesterol. We normalize $\varphi$ such that $\varphi = 1$ represents the pure saturated lipid phase and $\varphi=-1$ within the pure unsaturated lipid phase. Moreover, $v=1$ and $u=1$ correspond to maximal saturation for the cholesterol concentrations.

The inclusion of the cholesterol concentration in the model is due to the fact that because of their structure, cholesterol molecules have a strong affinity for saturated lipids. Models for phase separation on cell membranes thus often include the cholesterol concentration as it is argued that the strong affinity between saturated lipids and cholesterols enables active cellular transport processes of cholesterol to influence the phase separation process, leading to the formation of lipid rafts. We refer the reader to corresponding discussions in \cite{FSH,RL,GSR,Fo} and of course \cite{GKRR}. That the formation of lipid rafts is linked to the presence of cholesterols has also been observed in experiments, see for example \cite{LR}.

Let now for $\eps > 0$
\begin{equation}
\label{eq:energy}
\F(v, \varphi) = \int_\Gamma\frac\eps2|\nabla\varphi|^2 +
\eps^{-1}W(\varphi) + \frac1{2}(2v - 1 - \varphi)^2 \, d\mathcal{H}^2,
\end{equation}
with the double-well potential $W(s)=(1-s^2)^2.$ The functional $\F$ consists of two parts. The first part $\int_\Gamma\frac\eps2|\nabla\varphi|^2 +
\eps^{-1}W(\varphi)\, d\mathcal{H}^2$ is a classical Ginzburg-Landau energy, modeling the phase separation between the two lipid phases. The second part $\frac1{2}\int_{\G}(2v - 1 - \varphi)^2\, d\mathcal{H}^2$ accounts for the affinity between saturated lipid molecules and membrane-bound cholesterol.

We now assume that the evolution of the membrane quantities is driven by chemical potentials derived from the functional $\F.$ Namely, we introduce
\begin{align*}
\mu &:= \frac{\delta \F}{\delta \varphi} 
=  - \eps \SL \varphi + \eps^{-1}W'(\varphi) - (2v - 1 - \varphi),\\
\theta &:= \frac{\delta \F}{\delta v} = 2(2v - 1 - \varphi),
\end{align*}
and say that $\F$ is the surface free energy functional of the model.

We then consider the following bulk--surface system consisting of a surface Cahn--Hilliard equation coupled to a bulk--diffusion equation,
\begin{alignat}{2}
\label{eq:diffU}
\partial_t u &=  \Delta u &\qquad \text{in } B \times (0,T]&,\\
\label{eq:flux}
-  \nabla u \cdot \nu & = q &  
\text{on } \Gamma \times (0,T]&,\\
\label{eq:CH1}
\pd_t \varphi &= \SL \mu
&\text{on } \Gamma \times (0,T]&,\\
\label{eq:CH2}
\mu &= - \eps \SL \varphi + \eps^{-1}W'(\varphi) - (2v - 1 - \varphi)
&\text{on } \Gamma \times (0,T]&,\\
\label{eq:v}
\pd_t v &= \SL \theta + q = 4\SL v - 2\SL
\varphi + q
&\text{on } \Gamma \times (0,T]&\\
\theta &=  2(2v - 1 -
\varphi)&\text{on } \Gamma \times (0,T]& \label{eq:theta}
\end{alignat}
with initial conditions for $u$, $\varphi$ and $v$. Here we denote by $\nu$ the outer unit normal vector of $B$ on $\G.$

Let us comment on the meaning of the equations. The equations \eqref{eq:CH1} and \eqref{eq:v} describe the mass balance equations for the surface quantities. Moreover, \eqref{eq:diffU} and \eqref{eq:flux} model the evolution of the cytosolic cholesterol by diffusion. An important aspect is the inclusion of Neumann boundary conditions for the cytosolic diffusion. In dependence of the exchange term $q$, the cholesterol flux from the cytosol $B$ onto the membrane $\G$ appears as a source term for the evolution of the membrane-bound cholesterol $v$ in \eqref{eq:v}. Let us note that \eqref{eq:v} also includes a cross-diffusion, which stems from the cholesterol-lipid affinity in the surface energy $\F$. Finally, \eqref{eq:CH1} and \eqref{eq:CH2} model a Cahn-Hilliard dynamics for the lipid concentration and allow for a contribution from the cholesterol evolution via the last term. 

The discussion in \cite{GKRR} shows that the model is thermodynamically consistent for arbitrary constitutive choices for the exchange term $q$. Moreover, numerical simulations carried out in \cite[Section 5]{GKRR} illustrate how different constitutive choices for $q$ influence the qualitative behaviour of the coupled system. In particular, the model features the emergence of micro domains (or lipid rafts) for certain constitutive choices for $q.$ 

Our interest in this contribution is the sharp-interface limit $\eps\searrow 0$ in the lipid raft model. Formal results on this singular limit were already obtained by Garcke, Rätz, Röger and the second author in \cite[Section 4]{GKRR}. These results rely on formally matched asymptotics, where it is a priori assumed that solutions $(u_\varepsilon, \varphi_\varepsilon, v_\varepsilon, \mu_\varepsilon, \theta_\varepsilon)$ to the model \eqref{eq:diffU} -- \eqref{eq:theta} formally converges to a limit $(u,\varphi, v, \mu, \theta)$ as $\varepsilon \searrow 0$ and that for each $t \in (0,T]$ the zero level set $\lbrace \varphi_\varepsilon(\cdot,t) = 0 \rbrace$ converges to a smooth curve $\gamma(t) \subset \Gamma.$ Moreover, the technique assumes that all functions $(u_\varepsilon, \varphi_\varepsilon, v_\varepsilon, \mu_\varepsilon, \theta_\varepsilon)$ admit suitable expansions in $\varepsilon,$ both in a neighborhood of the interface and away from the interface. We refer the reader to \cite[Section 4]{GKRR} for more details and to \cite{Fi88, CF, AHM, NMHS} as well as \cite{NAY, KC} for a general introduction while acknowledging that this is by far not a comprehensive list of references.  

The sharp interface model obtained from the formal asymptotic analysis is given by 
\begin{alignat}{2}
\varphi &= \pm 1  &\quad &\text{on } \Gamma^{\pm}(t), t\in (0,T] ,\label{eq:sharp_a}\\
\partial_t u &= D \Delta u \qquad &&\text{in } B \times (0,T],\label{eq:sharp_b}\\
- D \nabla u \cdot \nu & = q  \qquad &&\text{on } \Gamma\times (0,T] ,\label{eq:sharp_c}\\
\Delta_\Gamma \mu &= 0  &&\text{on } \Gamma^{\pm}(t),t\in  (0,T],\label{eq:sharp_d}\\
\partial_t v &= \Delta_\Gamma \theta + q  &&\text{on } \Gamma^{\pm}(t), t\in (0,T],\label{eq:sharp_e}\\
\theta &= 2 \left(2v-1 \mp 1\right)  &&\text{on } \Gamma^{\pm}(t),t\in (0,T],\label{eq:sharp_f}\\
2 \mu + \theta&= c_0 \kappa_g  &&\text{on } \gamma(t), t\in (0,T] ,\label{eq:sharp_g}\\
\jump{\mu} &= 0  &&\text{on } \gamma(t),t\in (0,T] ,\label{eq:sharp_h}\\
\jump{\theta} &= 0  &&\text{on } \gamma(t),t\in (0,T] ,\label{eq:sharp_i}\\
- 2 \mathcal{V} &= \jump{\nabla_\Gamma \mu}\cdot\nu_\gamma  &&\text{on } \gamma(t),t\in (0,T] ,\label{eq:sharp_j}\\
- \mathcal{V} &= \jump{\nabla_\Gamma \theta}\cdot\nu_\gamma  &&\text{on } \gamma(t), t\in (0,T], \label{eq:sharp_k}
\end{alignat}
where $\jump{f}(x_0,t_0))=\lim_{h\to 0+} (f(x_0+\nu_\gamma(x_0,t_0)- f(x_0-\nu_\gamma(x_0,t_0))$ is the jump of a quantity $f$ across the interface $\gamma(t):= \partial \Gamma^+(t)$ and $\nu_\gamma(x_0,t_0) \in T_{x_0}\Gamma$ denotes the unit normal to $\gamma(t_0)$ in $x_0\in\gamma(t_0)$, pointing inside $\Gamma^+(t_0)$. The geodesic curvature of $\gamma(t)$ in $\Gamma$ is denoted by $\kappa_g(\cdot,t)$ and $\mathcal{V}(x_0,t_0)$ denotes the normal velocity of 
$\gamma(t_0)$ in $x_0 \in \gamma(t_0)$ in direction of $\nu_\gamma(x_0,t_0)$. For its precise definition, let  
$\gamma_t: U \to \gamma(t) \subset \Gamma$, $t\in (t_0-\delta,t_0+\delta)$ be a smoothly evolving family of local parameterizations of the curves $\gamma(t)$ by arc length over an open interval  $U \subset \mathbb{R}$ and let $\gamma_{t_0}(s_0)=x_0$ for some $s_0 \in U$. Then the normal velocity in $(x_0,t_0)$ is given by
\[ 
\mathcal{V}(x_0,t_0) =  \frac{d}{dt}\Big|_{t_0}\gamma_{t}(s_0) \cdot \nu_\gamma(x_0,t_0), 
\]
see also \cite{DDE}.

Apart from the formal technique discussed above, there are also rigorous results concerning singular limits of phase-field models and in particular of the Cahn-Hiliard equation. Matched asymptotic expansions can be justified rigorously in some cases (see e.g. \cite{ABC} for the Cahn-Hilliard equation). In contrast to these kind of results Chen employed ideas from geometric measure theory to prove that weak solutions to the Cahn-Hilliard equation approach weak solution of the Mullins-Sekerka equation \cite{ChenCH}. 

The main purpose of the present contribution is to extend Chen's result to the coupled bulk-surface system \eqref{eq:diffU}--\eqref{eq:theta} and thus proving rigorously the convergence of solutions to the lipid raft model to solutions of the sharp interface problem \eqref{eq:sharp_a}--\eqref{eq:sharp_k} as $\eps\searrow 0$.

Before stating our main result in full detail, we briefly introduce the concept of varifolds from geometric measure theory and fix some notation. This will help with the definition of a weak varifold solution to the sharp-interface problem \eqref{eq:sharp_a}--\eqref{eq:sharp_k}.

This paper is based on the work in the PhD thesis of the second author \cite{Diss}.

\subsection{Some Notation and Varifolds}\label{sec:gmt}
For a detailed presentation of geometric measure and varifold theory we refer the reader to the books of Simon \cite{LS}, Federer \cite{HF}, and the more accessible book by Morgan \cite{FMo} as well as the paper by Allard \cite{WA}. 

 In the sharp-interface limit which is our main application, the object in question will be a generalization of a curve on a two dimensional submanifold of $\R^3.$ Therefore for we assume from now on that $\X$ is a $l-$dimensional submanifold of $\R^n.$ Of course, this requires $l \leq n.$ 

 The symbol $\Sigma$ will always denote a $\sigma-$algebra on $\X.$ The Borel $\sigma$-algebra (i.e. the smallest $\sigma-$algebra that contains all open subsets of $\X$) will be denoted by $\mathcal{B}(\X)$.

For a measure $\lambda: \Sigma \rightarrow [0,\infty]$ and a measure space $(\X,\Sigma,\lambda),$ we recall that $\lambda$ is a Borel measure if $\mathcal{B}(\X) \subset \Sigma$ and that a inner regular Borel measure which is finite on compact sets is called Radon measure. A measure $\lambda: \Sigma \rightarrow (-\infty,\infty)$ is called a signed measure.	We denote the space of all signed Radon measures on $\X$ by $\M(\X).$

For any signed measure $\lambda$ the variation measure $\abs{\lambda}: \Sigma \rightarrow [0,\infty]$ is defined by
\[ \abs{\lambda}(A) := \sup\left\lbrace \left. \sum_{k\in\N} \abs{\lambda(A_k)} \right| \bigcup_{\kappa\in\N} A_k \subset A, A_k \bigcap A_l = \emptyset \text{ for all } k \neq l \right\rbrace. \]

By the Riesz Representation Theorem \cite[Theorem 1.54]{AFP}, the space $\M(\X)$ can also be characterized as a dual space of $C_0(\X),$ where $C_0(\X)$ is the closure of $C_c^\infty(\X)$ with respect to the supremums norm. Convergence in the sense of measures will always be the weak$-\ast$ convergence in $\M(\X)=\left(C_0(\X)\right)'.$

The aim of this section is to introduce the notion of a general varifold as a measure theoretic generalization of a submanifold.  By $\mathbb{S}^k(p)$ we denote for $k \leq l$ the Grassmanian of all $k-$dimensional subspaces of $T_p(\X)$ 
\[ \mathbb{S}^k(p) := \left\lbrace S \left| S  \text{ is a } k-\text{dimensional subspace of } T_p(\X)  \right. \right\rbrace.  \]
To introduce a topology on $\mathbb{S}^k(p),$ we introduce \[ V^kT_p(\X):= \left\lbrace \left. (v_1,\ldots,v_k) \right| v_1,\ldots, v_l \in T_p(\X) \text{ and linearly independent}\right\rbrace. \]
The topology on $\mathbb{S}^k(p)$ is then given as the quotient topology induced by the map
\[ \pi: V^kT_p(\X) \rightarrow \mathbb{S}^k(p)  \]
which sends a tuple of $k$ linearly independent vectors in $T_p(\X)$ onto the $k-$dimensional subspace they span.  

Moreover, we define $G_k(\X)$ as 
\[ G_k(\X) :=\left\lbrace (p,S) \left| p \in \X, S \in \mathbb{S}^k(p) \right. \right\rbrace. \]
Since $G_k(\X)$ is (at least locally for $U\subset \X$) diffeomorphic to $U\times\mathbb{S}^k(p),$ the topologies on $\X$ and $\mathbb{S}^k(p)$ induce a topology on $G_k(\X),$ see for example \cite[Lemma 2.2]{Lee}.
    
\begin{definition}[Varifold]
Let $\X \subset \R^n$ be an $l$-dimensional Riemanian manifold and let $G_k(\X)$ be defined as above. A general $k-$varifold (varifold for short) on $\X$ is a Radon measure on $G_k(\X).$  	
\end{definition}
\begin{remark}It is useful to introduce the following notation.
	\begin{enumerate}
		\item{The orthogonal projection of $T_p\X$ onto $S \in \mathbb{S}^k(p)$ will also be denoted by $S.$}
		\item{For $S \in \mathbb{S}^k(p)$ we denote by $\delta_{S}$ the Dirac measure concentrated on $S.$ That is, for a set $P \subset \mathbb{S}^k(p)$ we define \[ \delta_S(P):=\begin{cases}
			1, \quad \text{if } S \in P \\
			0, \quad \text{else.}
			\end{cases}\]}
		\item{Let $\X$ be an $l-$dimensional manifold. We identify $\mathbb{S}^{l-1}(p) \cong S_{l-1}(p)\mod\lbrace e_1, -e_1 \rbrace $ where $S_{l-1}(p)$ is the $(l-1)-$sphere in $T_p(\X)$ and $e_1$ is the first unit vector. As such, we identify $\mathbb{S}^{l-1}$ with the set of all unit normal vectors to unoriented $(l-1)$ planes in $T_p\X.$}
	\end{enumerate}
\end{remark}
\begin{definition}[Weight Measure of a Varifold]
	Let $V$ be an varifold on $\X.$ The measure $m_V$ on $\X$ defined as
	\[ m_V(A) := V(\lbrace \left.(p,S)\right| p \in A, S \in \mathbb{S}^k(p) \rbrace) = \int_{G_k(A)} \ dV(p,S) \]
	is called weight measure of $V.$
\end{definition}
\begin{definition}[First Variation of a Varifold]
	Let $V$ be an $(l-1)$-varifold on $\X.$ The first variation $\delta V: C^1_c(\X,T\X) \rightarrow \R$ is given as
	\[ \delta V (\xi) = \int_{G_k(\X)} D_\X \xi(p) : \left(\Id - S\otimes S\right) \ dV(p,S) \]
	for all $\xi \in C^1_c(\X,T\X).$ Here $D_\X \xi(p)$ denotes the differential of $\xi$ on $\X$, $A:B= \operatorname{tr} (A^\ast B)$ for all $A,B\in \mathcal{L}(T_x\X,\R^n)$ and $\Id$ denotes the identity on $T_x\X$ for each $x\in\X$.
\end{definition}

\section{Weak Solutions and Statement of the Main Results}\label{ch:asymptotics}

Throughout this paper, we consider the convergence of weak solutions $(u_\varepsilon,\varphi_{\varepsilon},v_{\varepsilon},\mu_\varepsilon,\theta_\varepsilon)$ to the diffuse interface problem \eqref{eq:diffU}--\eqref{eq:theta} to a weak (varifold) solution to the sharp interface problem \eqref{eq:sharp_a}--\eqref{eq:sharp_k} as $\varepsilon \searrow 0.$ In the case that the exchange term growth at most linearly, i.e. that $q: \R^2 \rightarrow \R$ is continuous and fulfils for some $C>0$
\begin{equation}\label{eq:growth_cond_q}
\abs{q(u,v)}\leq C(1+\abs{u}+\abs{v}) \qquad \forall \ u,v \in \R,
\end{equation} the existence of weak solutions $(u_\varepsilon,\varphi_{\varepsilon},v_{\varepsilon},\mu_\varepsilon,\theta_\varepsilon)$ to the diffuse interface problem is granted by Theorem \cite[Theorem 2.3]{Preprint1} or \cite[Theorem~4.2]{Diss} where it is proved that there are solutions belong to the space
\[ \solspace := \solspace_B \times \solspace_\Gamma^1 \times \solspace_\Gamma^1\times \solspace_\Gamma^2\times \solspace_\Gamma^2, \]
where
\begin{align*}
\solspace_B  &:= L^2\left(0,T;H^1(B)\right) \cap  H^1\left(0,T;\left(H^{1}(B)\right)'\right),\ \\ 
\solspace_\Gamma^1  &:= L^2(0,T;H^1(\Gamma))\cap H^1(0,T;H^{-1}(\Gamma)),  \text{ and }
\solspace_\Gamma^2  := L^2(0,T;H^1(\Gamma)).
\end{align*}
Here the equations are solved in the following weak sense:
\begin{align}
  \int_0^T \left\langle \partial_t u, \xi \right\rangle_{\left(H^1(B)\right)',H^1(B)} &= -\int_0^T\int_B \nabla u \cdot \nabla \xi - \int_0^T\int_{\Gamma} q(u,v)\xi,\label{eq:weak_form_diffU} \\
  \int_0^T \dprodH{\G}{\partial_t \varphi}{\eta} &= -\int_0^T\int_{\Gamma} \SG \mu \cdot \SG \eta, \label{eq:weak_form_time_der_phi} \\
  \int_0^T\int_{\Gamma} \mu \eta &= - \int_0^T\int_{\Gamma} \left[\varepsilon \SG \varphi \cdot \SG \eta + \frac{1}{\varepsilon} W'(\varphi)\eta -\frac{1}{\delta}\left( 2v - 1 -\varphi \right)\eta \right],\\
  \int_0^T \dprodH{\G}{\partial_t v}{\eta} &= -\int_0^T\int_{\Gamma} \SG \theta \cdot \SG \eta +\int_0^T\int_{\Gamma} q(u,v)\eta, \\
  \theta  &= \frac{2}{\delta} \left( 2v - 1 - \varphi\right) \text{ a.e. on } \G\times(0,T).\label{eq:weak_form_theta}
\end{align}
for all $\xi \in L^2(0,T;H^1(B))$ and $\eta \in L^2(0,T;H^1(\Gamma))$.
The initial values are attained in $L^2(B)$ and $L^2(\G)$ respectively.
Moreover,
\begin{align}
  \F(v(\cdot,t),\varphi(\cdot,t)) + \frac{1}{2}\int_B u(\cdot,t)^2
  + \int_0^t\int_B \frac{1}{2}|\nabla u|^2 + \int_0^t \int_{\Gamma} \left( \abs{\nabla_\Gamma \mu(\cdot,t)}^2 + \abs{\nabla_\Gamma \theta(\cdot,t)}^2 \right) 
  \, \nonumber \\ \leq\, C(T,v_0,\varphi_0,u_0). \label{eq:energy_est_lin_growth}
\end{align} holds for all $t \in (0,T],$ see also \cite[Theorem 2.3]{Preprint1}.

Following \cite{ChenCH}, we define weak (varifold) solutions to the sharp interface problem as follows. 

\begin{definition}\label{def:weak_sol_varifold}
	Let $E$ be a subset of $\Gamma\times[0,\infty)$ and assume that $\chi_E \in C^0([0,T);L^1(\Gamma))\cap L_{w^\ast}^\infty(0,T;BV(\Gamma)).$ Consider functions \[ \mu,\theta \in L^2_{loc}([0,T),H^1(\Gamma)) \] and \[ u \in H^1(0,T;H^{-1}(B))\cap L^2(0,T;H^1(B)), v \in H^1(0,T;H^{-1}(\Gamma))\cap L^2(0,T;H^1(\Gamma)).\] Let furthermore $V$ be a Radon measure on $[0,\infty)\times G_1(\G)$ such that $V_t$ is a varifold on $\Gamma$ for all $t\geq0.$
	
	We say that the tuple $(E,V,u,\mu,\theta)$ is a varifold solution to the sharp interface problem \eqref{eq:sharp_a}--\eqref{eq:sharp_k} if for all $T\geq0$ and for almost every $0\leq\tau\leq t \leq T$ and for all test functions \[ \psi_b \in C^\infty_c([0,T)\times\overline{B}), \psi_s \in C^\infty_c([0,T)\times\Gamma) \text{ and } Y \in C^1(\Gamma,T\G) \] the following holds:
	\begin{align} 
	\int_0^T \int_B &u(t,x) \partial_t \psi_b(t,x) \ dx \ dt = \\\nonumber  &\int_B u_0(x) \psi_b(0,x) \ dx + \int_0^T \int_B \nabla u(t,x) \cdot \nabla \psi_b(t,x) \ dx \ dt - \int_0^T \int_{\partial B} q \psi_b \ d\mathcal{H}^2(p) \ dt, \\ 
	\int_0^T \int_\Gamma&  -2 \chi_{E_t}(p) \partial_t \psi_s(t,p) + \SG \mu(t,p) \cdot \SG \psi_s(t,p) \ d\mathcal{H}^2(p) \ dt \nonumber \\ &= \int_\Gamma 2\chi_{E_0}(p) \psi_s(0,p) \ d\mathcal{H}^2(p), \\
	\int_0^T \int_\Gamma &-v(t,p)\partial_t \psi_s(t,p) + \SG \theta(t,p) \cdot \SG \psi_s(t,p) - q\psi_s(t,p) \ d\mathcal{H}^2(p) \ dt \nonumber \\ & = \int_\Gamma v_0(p) \psi_s(0,p) \ d\mathcal{H}^2(p), \\ 
	&\theta =  2 \left( 2v - 2 \chi_{E_t} \right),  \\
	-&\langle D\chi_{E_t}, (2\mu+\theta) Y \rangle = \langle \delta V_t, Y\rangle, \label{eq:weak_varifold_formulation_curv} \\
	&m_{V_t} \geq 2c_0\abs{D\chi_{E_t}}\qquad \text{in }\mathcal{M}(\Gamma), \\\label{eq:Energy}
	m_{V_t}(\Gamma) + &\int_\tau^t \int_\Gamma \abs{\SG\mu(s,p)}^2 + \abs{\SG\theta(s,p)}^2  + q\left(\theta(s,p)-u(s,p)\right) \ d\mathcal{H}^2(p) \ ds \leq m_{V_\tau}(\Gamma).
	\end{align}
\end{definition}
\begin{remark}
	The concept of a varifold solution given here coincides in the special case that $u=v=0$ with the varifold solutions introduced by Chen in \cite{ChenCH}. We refer the reader to \cite[Section 2.4]{ChenCH} for a detailed discussion of these solutions and a justification of the definition.
      \end{remark}

Our main convergence result establishes problem \eqref{eq:sharp_a}--\eqref{eq:sharp_k} rigorously as the sharp interface limit of the lipid raft model (for a suitable subsequence) if we suppose that the initial data are suitable in the following sense.
\begin{condition}[Assumptions for the inital data]\label{cond:well_prep_data}~\\
	We assume that there exist constants $C,M,m > 0$ and independent of $\eps\in (0,1]$ such that the initial data $(u^\eps_0, \varphi^\eps_0,v_0^\eps)\in L^2(B)\times H^1(\Gamma)\times H^1(\Gamma)$ fulfil
	\begin{alignat*}{2}
	\sup_{0<\eps<1} \left[ \F(\varphi^\eps_0, v_0^\eps) + \int_B \abs{u_0^\eps}^2 \ dx\right] &\leq C < \infty, &\\
	\int_B u_0^\eps \ dx + \int_{\G} v_0^\eps \ d\mathcal{H}^2 &= M &\forall \eps \in (0,1],\\
	\frac{1}{\abs{\G}} \int_{\G} \varphi_0^\eps \ d\mathcal{H}^2 &=m \in (-1,1) \  &\forall \eps \in (0,1].
	\end{alignat*}
\end{condition}
\begin{proposition}\label{prop:collected_conv}
	Let $T>0$ and consider initial data that fulfil Condition \ref{cond:well_prep_data} and the corresponding solution $(u_\eps,\varphi_{\eps},v_{\eps},\mu_\eps,\theta_\eps) \in \solspace$ to the diffuse interface problem \eqref{eq:diffU}--\eqref{eq:theta}. Then there exists a sequence $\lbrace \varepsilon_k \rbrace_{k\in\N}, \varepsilon_k \rightarrow 0$ as $k\rightarrow\infty,$ such that the following statements are true:
	\begin{enumerate}
		\item{There exists a set $Q^+ \subset [0,T) \times \Omega$ such that 
			\begin{enumerate}
				\item $\varphi_{\varepsilon_k}(x,t) \rightarrow \varphi(x,t):= \begin{cases} 1 \ \ \text{ for } (x,t) \in Q^+ \\ -1 \text{ else } \end{cases}$ almost everywhere in $(0,T) \times \Gamma.$
				\item $\varphi_{\varepsilon_k} \rightarrow \varphi$ in $C^{1/9}\left( [0,T]; L^2(\Gamma) \right).$
				\item $\chi_{Q^+} \in L_{w^\ast}^\infty\left(0,T; BV(\Gamma) \right).$
		\end{enumerate}}
		\item There exists a function $\mu \in L^2(0,T;H^1(\Gamma))$ such that \[ \mu_\epsk \rightharpoonup \mu \text{ in } L^2(0,T;H^1(\Gamma)).\]
		\item There exists a function $\theta \in L^2(0,T;H^1(\Gamma))$ such that \[ \theta_\epsk \rightharpoonup \theta \text{ in } L^2(0,T;H^1(\Gamma)).\]
		\item There exists a function $v \in L^2(0,T;L^2(\Gamma))$ such that \[ v_\epsk \rightharpoonup v \text{ in } L^2(0,T;L^2(\Gamma)).\]
		\item There exists a function $u \in L^2(0,T;H^1(B))\cap H^1(0,T;L^2(B))$ such that \[ u_\epsk \rightharpoonup u \text{ in } L^2(0,T;H^1(B))\cap H^1(0,T;L^2(B)).\]
	\end{enumerate}
\end{proposition}
\begin{theorem}\label{thm:varifold_sol}
	Let $(u,\varphi,v,\mu,\theta)$ be the limit tuple from Proposition \ref{prop:collected_conv}. Then there exists a Radon measure $V$ on $[0,T]\times G(\G)$ such that the measure $V_t := V(t,\cdot)$ is a varifold for almost all times $t \in [0,T].$  Moreover, the tuple $(\varphi,\mu,v,\theta,u,V)$ is a weak solution to the sharp interface problem \eqref{eq:sharp_a}--\eqref{eq:sharp_k} in the varifold sense defined in Definition \ref{def:weak_sol_varifold}. 
\end{theorem}

\section{Proof of the main convergence results}

The proof follows the arguments of the corresponding result by Chen for the Cahn-Hilliard equation \cite{ChenCH} with modification due to the coupling and the fact that the Cahn-Hilliard type equation is given on a surface $\Gamma$. The proof consists of two main parts. First, one establishes compactness results for the individual functions which lead to the convergences stated in Proposition \ref{prop:collected_conv}. In the second step, we take the limit in the equations and construct the varifold representing the phase boundary. The main ingredient will be an estimate for the so called discrepancy measure in Proposition \ref{prop:control_discr_measure}. It relies on a local blow-up argument with regard to $\eps$ and requires us to carefully consider equation \eqref{eq:CH2} in suitable local coordinates, making the blow-up procedure more delicate then in Chen's original work. Moreover, Chen's construction only yields varifolds in every local chart. We need to prove that these varifolds serve as a starting point for the construction of a varifold on $\G$ which then, together with the limit functions from the first part, solves the sharp interface problem \eqref{eq:sharp_a}--\eqref{eq:sharp_k} in the weak sense of Definition \ref{def:weak_sol_varifold}. 

 The first part requires only modest modifications, mostly in order to adopt some of the technical details to the setting on a compact manifold. It relies on the energy estimate \eqref{eq:energy_est_lin_growth}, the structure of \eqref{eq:CH2}, the mass conservation for $\varphi_\eps$ and the well-known Modica-Mortola trick which is included here for the sake of completeness. In particular, it allows us to deduce bounds for $\lbrace \varphi_{\varepsilon}\rbrace_{\varepsilon > 0}$ which are uniform in $\varepsilon.$  
\begin{lemma}[Modica-Mortola trick]\label{l:modica_mortola_trick}
	Let $H:[-1,\infty)\rightarrow [0,\infty)$ be defined by
	\begin{equation*}
	H(s)=\int_{-1}^s \sqrt{\min\lbrace W(s), 1+\abs{r}^2\rbrace} \ dr.
	\end{equation*}
	$H$ is invertible and there are constants $c_1, c_2 \in \R$ such that
	\begin{equation}\label{est:propH}
	c_1 \left| s_1 - s_2 \right|^2 \leq \left| H(s_1) - H(s_2) \right| \leq c_2 \left| s_1 - s_2\right| \left(1 + \left| s_1 \right| +\left| s_2 \right| \right) 
	\end{equation}
	for all $s_1,s_2 \in \mathbb{R}$. Moreover, for any $\varepsilon > 0$ and any solution $(u_\varepsilon,\varphi_\varepsilon,v_\varepsilon, \mu_\varepsilon,\theta_\varepsilon) \in \solspace$ to \eqref{eq:diffU}--\eqref{eq:theta} 
	\begin{equation}\label{eq:est_modica_mortola_trick}
	\sup_{0\leq t \leq T} \int_\G \left| \SG H(\varphi_\varepsilon(x,t)) \right| \leq C(T)
	\end{equation} 
	and in particular
	\[ H(\varphi_\varepsilon) \in L^\infty(0,T;W^{1,1}(\Gamma)). \]	
\end{lemma}
\begin{proof}
	The existence of constants $c_1$ and $c_2$ such that \eqref{est:propH} holds is a direct consequence from the properties of the double-well potential $W.$ Since the integrand is positive, the function $H$ is strictly monotonically increasing and thus invertible. 
	
	We calculate 
	\begin{align*}
	\int_\G \left| \SG H(\varphi_\varepsilon(x,t)) \right| & = \int_\G \sqrt{W(\varphi_\varepsilon(\cdot,t))} \left|\SG \varphi_\varepsilon(\cdot, t) \right| \nonumber \\ & \leq \frac{1}{2} \int_\G \left[ \frac{1}{\varepsilon}W(\varphi_\varepsilon(\cdot,t)) + \varepsilon \left| \SG \varphi_\varepsilon(\cdot, t) \right|^2 \right]
	\end{align*}
	which by the energy estimate \eqref{eq:energy_est_lin_growth} implies \eqref{eq:est_modica_mortola_trick} and $H(\varphi_\varepsilon) \in L^\infty(0,T;W^{1,1}(\Gamma))$.
\end{proof}

\subsection{Preliminary results}\label{sec:prelim}

We quickly recall that the exponential map $\exp_p$ from differential geometry in a point $p \in \G$ is a diffeomorphism between an open neighborhood $W_p \subset T_p \G$ and an open neighborhood $U \subset \Gamma$ of $p.$ Since $\G$ is a compact manifold, there is a real number $r > 0$ such that the ball $B_r(0)$ lies in $W_p$ for all $p \in \G$ and the map $\exp_p$ restricted to $B_r(0)\subset\R^2$ is a diffeomorphism onto its image for all $p\in\G$.

Thus the sets $\lbrace \left.\exp_p\right|_{B_{r-\eta}(0)}(B_{r-\eta}(0)) \rbrace_{p\in\G}$ form for every $\eta \leq \frac{r}{2}$ a covering of $\G.$ Since the manifold $\G$ is compact, there is a finite collection of points $\lbrace p_i\rbrace_{i\in\IS{I}}$ such that $\lbrace U_i \rbrace_{i\in\IS{I}} := \lbrace \left.\exp_{p_i}\right|_{B_{r-\eta}(0)}(B_{r-\eta}(0)) \rbrace_{i\in\IS{I}}$ still is a covering of $\G.$ Together with the maps $\alpha_i: U_i \rightarrow B_{r-\eta}(0)\subset\R^n$ defined by $\alpha_i(p) := \exp_{p_i}^{-1}(p),$ this covering allows us to define an atlas $\lbrace (U_i,\alpha_i)\rbrace_{i\in\IS{I}}$ of $\G.$

Observe that for every $i \in \IS{I}$ and $x\in \alpha_i(U_i)=B_{r-\eta}(0)\subset\R^n$, the expression $\exp_{p_i}(x-\eta y)$ is well defined as long as we assume $y \in B_1(0).$ We will make use of this fact in the following construction of approximating sequences to functions in $L^2(\Gamma).$ 

Let $\rho$ be a mollifier satisfying
\[
  \rho \in C^\infty(\R^2), \qquad 0 \leq \rho(x) \leq 1 \forall x \in \R^2, \qquad \supp(\rho)\subset B_1(0) \text{ and } \int_{\R^2} \rho = 1
\]
as usual and introduce the notation $\rho_\eta(x) = \eta^{-2} \rho\left(\frac{x}{\eta}\right).$ Furthermore, let $\lbrace z_i \rbrace_{i\in\IS{I}}$ be a partition of unity subordinate to the covering $\lbrace U_i \rbrace_{i\in\IS{I}}$ of $\Gamma.$ 

For a function $v\in\L{2}$ we can then define the functions $v_i^\eta \in C^\infty(\alpha_i(U_i))$ for every $i\in\IS{I}$ and $\eta \leq \frac{r}{2}$ by
\begin{align*} 
v_i^\eta(x) &:= \left(\rho_\eta*(\alpha_i^{-1,*}(z_iv))\right)(x) \nonumber \\
&\ = \int_{B_\eta(0)} \rho_\eta(y)(z_iv)(\alpha^{-1}_i(x-y)) \ dy = \int_{B_1(0)} \rho(y)(z_iv)(\alpha^{-1}_i(x-\eta y)) \ dy.
\end{align*}
We deduce $\supp{\alpha^{-1,*}(z_iv)} \subset B_{r-\eta}$ from $\supp{z_iv} \subset U_i$ and since $\supp{\rho_\eta} \subset B_\eta(0)$ this implies $\supp{v_i^\eta}\subset B_r(0)$ by the general properties of convolutions. Thus the pullback $\alpha_i^* v_i^\eta$ is well defined for each $i \in \IS{I}$ and we can define for $\eta \leq \frac{r}{2}$ the smoothing operator $T_\eta: \L{2} \rightarrow C^\infty(\Gamma)$ by
\begin{equation*}
T_\eta v := \sum_{i\in\IS{I}} \alpha_i^* v_i^\eta = \sum_{i\in\IS{I}} \alpha_i^*\left(\rho_\eta*(\alpha_i^{-1,*}(z_iv))\right)
\end{equation*}

\begin{lemma}[Approximation on manifolds]\label{l:approx_on_manifolds}
	For each $v \in L^2(\Gamma),$ the family $\lbrace v_\eta\rbrace_{0<\eta<r/2}$  defined by $v_\eta := T_\eta v$ is a smooth approximation of $v$ with respect to the $L^2$-topology, i.e.
	\begin{equation*}
	\norm{v_\eta - v}_{L^2(\Gamma)} \rightarrow 0 \text{ as } \eta \rightarrow 0.
	\end{equation*}
	Furthermore,
	\begin{align}\label{eq:approx_bounded}
	&\norm{v_\eta}_{L^2(\G)} \leq C \norm{v}_{L^2(\G)} \text{ and } \\
	\label{eq:approx_grad_rate}
	&\norm{\nabla_\Gamma v_\eta}_{L^2(\Gamma)} \leq \frac{C}{\eta} \norm{v}_{L^2(\Gamma)} 
	\end{align}
	for some constant $C > 0,$ independent of $\eta.$
\end{lemma}
\begin{proof}
	We first prove \eqref{eq:approx_bounded}. Observe that for all $i\in\IS{I}$ we can estimate (see for example \cite{ST13})
	\[ \norm{\alpha_i^* v_i^\eta}_{L^2(U_i)} \leq C \norm{v_i^\eta}_{L^2(U_i)} \]
	and vice versa
	\[ \norm{\alpha_i^{-1,*}(z_i v)}_{L^2(\alpha_i(U_i))} \leq C \norm{(z_i v)}_{L^2(U_i)}, \]
	where the constants $C$ are independent of $i\in\IS{I}$ since $\G$ is compact.
	Furthermore,
	\[ \norm{\rho_\eta*(\alpha_i^{-1,*}(z_iv))}_{L^2(\alpha_1(U_i))} \leq \norm{\rho_\eta}_{L^1(\R^2)} \norm{\alpha_i^{-1,*}(z_iv)}_{L^2(\alpha_1(U_i))} \leq \norm{\alpha_i^{-1,*}(z_iv)}_{L^2(\alpha_1(U_i))} \] by the usual properties of the convolution. Combining these findings yields \eqref{eq:approx_bounded}.
	
	The next claim is that $ v_\eta \rightarrow v$ in $L^2(\Gamma)$. Since $C(\Gamma)$ is a dense subset of $L^2(\Gamma)$ and because of \eqref{eq:approx_bounded}, it is sufficient to prove the convergence only for functions $v \in C(\Gamma).$
	
	One has
	\[ \sum_{i\in\mathcal{I}} \alpha_i^*\left(\int_{\R^n} \rho_\eta(y) \alpha_i^{-1,*}(z_i) \ dy \right)(p) = 1 \] for all $p  \in \Gamma.$ 
	Therefore
	\begin{align*}
	(v-v_\eta)(p) 
	&= v(p) \sum_{i\in\mathcal{I}} \alpha_i^*\left(\int_{\R^n} \rho_\eta(y) \alpha_i^{-1,*}(z_i) \ dy \right)(p) - \sum_{i\in\mathcal{I}} \alpha_i^*\left(  \rho_\eta * \alpha_i^{-1,*}(z_i v)\right)(p) \nonumber \\
	&= \sum_{i\in\mathcal{I}} \left[\int_{\R^2} \rho_\eta(y) \lbrace (z_iv)(\alpha_i^{-1}(\alpha_i(p)))-(z_iv)(\alpha_i^{-1}(\alpha_i(p)-y)) \rbrace \ dy \right].
	\end{align*}
	We now split the integrals into two parts, namely the integrals over a ball of radius $\xi$ around the origin and the integral over all $\abs{y}\geq\xi.$ To simplify the expression, we denote the integrand in the last expression by $I_{i}(y,p)$ and write thereby
	\begin{align*}
	(v-v_\eta)(p) &= \sum_{i\in\mathcal{I}}  \int_{\abs{y} < \xi} I_{i}(y,p) dy + \sum_{i\in\mathcal{I}} \int_{\abs{y} \geq \xi} I_{i}(y,p) dy \nonumber \\
	&=: \sum_{i\in\mathcal{I}} A^{i}_\eta(\xi) + \sum_{i\in\mathcal{I}} B^{i}_\eta(\xi).
	\end{align*}
	We now exploit the fact that $\alpha_i^{-1,*}(z_iv) \in C_c(\R^n)$ is uniformly continuous to deduce for every $\varepsilon > 0$ the existence of a radius $\xi_{i} > 0$ such that 
	\begin{align*}
	\abs{A_\eta^{i}(\xi_{i})} \leq \sup_{\abs{y}<\xi_{i}} \abs{ (z_iv)(\alpha_i^{-1}(\alpha_i(p)))-(z_iv)(\alpha_i^{-1}(\alpha_i(p)-y)) } \int_{\R^n} \rho_\eta (y) \ dy \leq \varepsilon. 
	\end{align*}
	Since the manifold $\Gamma$ is compact, we can define $\xi_0>0$ as the minimum over all $\xi_{i}.$ By the properties of $\rho_\eta,$ it is then possible to find $\eta_0 > 0$ such that
	\begin{align*}
	\abs{B^{i}_\eta(\xi_0)} \leq 2 \norm{v}_\infty \int_{\abs{y}\geq \xi_0} \rho_\eta(y) \ dy \leq C\varepsilon
	\end{align*}
	for all $\eta < \eta_0.$ These two estimates thus imply for every $\varepsilon > 0$ the existence of $\eta_0 > 0$ such that 
	\[ \norm{v-v_\eta}_\infty \leq C \varepsilon \text{ for all } \eta < \eta_0, \] which is the desired convergence.
	We now prove the second assertion of the lemma, namely the estimate \eqref{eq:approx_grad_rate}. Since   
	\begin{align}\label{eq:lin_grad_approx}
	\norm{\nabla_\Gamma v_\eta}_{L^2(\Gamma)} = \norm{\sum_{i\in\mathcal{I}}\nabla_\Gamma \left( \alpha_i^*\left(\rho_\eta * \alpha_i^{-1,*}(z_i v)\right)\right)}_{L^2(\Gamma)}
	\end{align}
	it is sufficient to estimate each summand on the right-hand side in \eqref{eq:lin_grad_approx}. We write the gradient on $\Gamma$ in local coordinates to obtain
	\begin{align*}
	\norm{\nabla_\Gamma \left( \alpha_i^*\left(\rho_\eta * \alpha_i^{-1,*}(z_i v)\right)\right)}_{L^2(\Gamma)}
	= \norm{\sum_{k,l}^2 g^{kl}\frac{\partial}{\partial x_k}\left(\rho_\eta * \alpha_i^{-1,*}(z_i v)\right)\partial_{x_l}\sqrt{g}}_{L^2(\alpha_i(U_i))}
	\end{align*}
	and use the fact that all entries in the metric tensor $g$ are bounded, first on each $U_i$ and then by the compactness of $\Gamma$ on the whole manifold, to see 
	\begin{align*}
	&\norm{\sum_{k,l}^n g^{kl}\frac{\partial}{\partial x_k}\left(\rho_\eta * \alpha_i^{-1,*}(z_i v)\right)\partial x_l\sqrt{g}}_{L^2(\alpha_i(U_i))}  \nonumber \\ \leq& C \norm{\sum_{k,l}^n \left(\frac{\partial}{\partial_{x_k}}\rho_\eta\right) * \alpha_i^{-1,*}(z_i v)}_{L^2(\alpha_i(U_i))} 
	\leq \frac{C}{\eta} \norm{\alpha_i^{-1,*}(z_i v)}_{L^2(\alpha_i(U_i))}
	\end{align*}
	where the last inequality is again due to Young's inequality for convolutions and the chain rule produced the factor $\frac{1}{\eta}.$ We use again the estimate
	\begin{align*}
	\norm{\alpha_i^{-1,*}(z_i v)}_{L^2(\alpha_i(U_i))} \leq C \norm{(z_i v)}_{L^2(U_i)}
	\end{align*}
	and deduce inequality \eqref{eq:approx_grad_rate}.
\end{proof}

\begin{proposition}\label{prop:schauder_laplace}
	For every $g \in C^1(\Gamma)$ with $\int_\Gamma g = 0$, there exists a solution $\Psi \in C^2(\Gamma)$ to the problem
	\begin{alignat*}{2}
	\Delta_\Gamma \Psi &= g  &\qquad \text{ on } \Gamma, \qquad
	\int_\Gamma \Psi = 0. 
	\end{alignat*}
	Furthermore, the estimate
	\begin{equation}\label{eq:testfunction_estimate}
	\norm{\Psi}_{C^2(\Gamma)} \leq C \norm{g}_{C^1(\Gamma)}
	\end{equation}
	holds.
\end{proposition}
\begin{proof}
  The result follows from classical regularity theory. Details can be found in \cite[Proposition~8.7]{Diss}.
\end{proof}

\subsection{Compactness results}
\subsubsection{Compactness of $\lbrace \varphi_\eps\rbrace$}

The main purpose of this paragraph is to prove the convergence of the family $\lbrace \varphi_\varepsilon \rbrace_{\varepsilon \geq 0}$. 
We expect the limit $\varphi := \lim_{\varepsilon \rightarrow 0} \varphi_\varepsilon$ to be a function which takes only the values $\pm 1$, although it is not directly clear in which sense this convergence is meant. 
\begin{proposition}\label{prop:conv_u}
	Under the assumptions from Proposition \ref{prop:collected_conv} there exists a set $Q^+ \subset [0,T) \times \G$ and a subsequence of $\lbrace \varphi_{\varepsilon} \rbrace_{\varepsilon> 0}$ (which we denote by $\lbrace \varphi_{\varepsilon_k} \rbrace$) such that 
	\begin{enumerate}
		\item $\varphi_{\varepsilon_k}(x,t) \xrightarrow{k \rightarrow \infty} \varphi(x,t):= \begin{cases} 1 \ \ \text{ for } (x,t) \in Q^+ \\ -1 \text{ else } \end{cases}$ almost everywhere in $(0,T) \times \G.$
		\item $\varphi_{\varepsilon_k} \xrightarrow{k \rightarrow \infty} \varphi$ in $C^{1/9}\left( [0,T]; L^2(\G) \right).$
		\item $\chi_{Q^+} \in L_{w^\ast}^\infty\left(0,T; BV(\G) \right).$
	\end{enumerate}
\end{proposition} 
The proof relies on the following lemma. 
\begin{lemma}\label{l:estChen}
	There exists a positive constant C which is independent of $\varepsilon$ such that
	\begin{equation*}
	\sup_{0\leq t \leq T}\norm{H(\varphi_\eps(t))}_{W^{1,1}(\G)} + \norm{H(\varphi_\eps)}_{C^{1/8}([0,T);L^1(\G))} + \norm{\varphi_\eps}_{C^{1/8}([0,T);L^2(\G))} \leq C.
	\end{equation*}
\end{lemma}
\begin{proof}
	The structure of the proof is the same as in \cite[Lemma 3.2]{ChenCH}. 
		
	We begin our proof with the estimate for $\norm{\varphi_\eps}_{C^{1/8}([0,T);L^2(\Omega))}.$ Our aim is to show that 
	\begin{equation}\label{eq:u_in_C}
	\sup_{t,\tau \in [0,T], t \neq \tau} \frac{\left( \int_\Omega \left| \varphi_\eps(x,t) - \varphi_\eps(x,\tau) \right|^2 \ dx \right)^{1/2}}{\left(t - \tau \right)^{1/8}}\leq C(T).
	\end{equation} 
	Since it is difficult to control this difference between $\varphi_\eps(x,t)$ and $\varphi_\eps(x,\tau)$ directly, we use the approximation result on the manifold $\G$ in Lemma \ref{l:approx_on_manifolds} and define $\varphi_\eps^\eta := T_\eta \varphi_\eps.$ 
	As in \cite{ChenCH}, we obtain
	\begin{align}\label{eq:chenLemma_help}
	&\int_\Gamma \left( \varphi_\eps(x,t) - \varphi_\eps(x,\tau) \right)^2 \ dx \nonumber \\ &\quad= \int_\Gamma \displaystyle{\left[ \left( \varphi_\eps^\eta(x,t) - \varphi_\eps^\eta(x,\tau) \right) - \left(\varphi_\eps(x,t) - \varphi_\eps(x,\tau)\right) \right]^2 \ dx} \nonumber \\
          & \qquad + 2 \int_\Gamma \left( \varphi_\eps^\eta(x,t)-\varphi_\eps^\eta(x,\tau)\right)\left( \varphi_\eps(x,t)-\varphi_\eps(x,\tau)\right) \ dx \nonumber \\
          & \qquad -\int_\Gamma \left( \varphi_\eps^\eta(x,t) - \varphi_\eps^\eta(x,\tau) \right)^2 \ dx \nonumber \\
	&\quad \leq \int_\Gamma \displaystyle{\left[ \left( \varphi_\eps^\eta(x,t) - \varphi_\eps(x,t) \right) + \left( \varphi_\eps(x,\tau) - \varphi_\eps^\eta(x,\tau)\right) \right]^2 \ dx} \nonumber \\
	& \qquad + 2 \int_\Gamma \left( \varphi_\eps^\eta(x,t)-\varphi_\eps^\eta(x,\tau)\right)\left( \varphi_\eps(x,t)-\varphi_\eps(x,\tau)\right) \ dx
	\end{align}
	since $\left( \varphi_\eps^\eta(x,t) - \varphi_\eps^\eta(x,\tau) \right)^2$ is non-negative. It is therefore sufficient to control the right-hand side above if we want to prove \eqref{eq:u_in_C}.
	
	To this end, we first observe that for $0<\alpha<\frac{1}{2}$ and any $t \in (0,T)$ 
	\begin{align*}
	&\int_\G \int_\G \frac{\abs{\varphi_\eps(x,t)-\varphi_\eps(y,t)}^2}{d(x,y)^{2+2\alpha}} \ d\mathcal{H}^2(x) \ d\mathcal{H}^2(y) \\
	&\quad \leq C \int_\G \int_\G \frac{\abs{H(\varphi_\eps(x,t))-H(\varphi_\eps(y,t))}}{d(x,y)^{2+2\alpha}} \ d\mathcal{H}^2(x) \ d\mathcal{H}^2(y).
	\end{align*}
	by the properties of $H$ discussed in Lemma \ref{l:modica_mortola_trick}. Moreover, for any function $f$ in the Besov space $B_{1,1}^{2\alpha}(\G)$ we have
	\begin{align*}
	\int_\G \int_\G \frac{\abs{f(x)-f(y)}}{d(x,y)^{2+2\alpha}} \ d\mathcal{H}^2(x) \ d\mathcal{H}^2(y) \leq C \norm{f}_{B_{1,1}^{2\alpha}(\G)}.
	\end{align*}
	Since $W^{1,1}(\G)$ embeds in the Besov space $B_{1,1}^{2\alpha}(\G)$ for $0<\alpha < \frac{1}{2}$ and $H(\varphi_\eps(\cdot,t)) \in W^{1,1}(\G)$ with $\norm{H(\varphi_\eps)(\cdot,t)}_{W^{1,1}(\G)} \leq C(T)$ for all $t \in (0,T)$ by Lemma \ref{l:modica_mortola_trick}, we can choose $f=H(\varphi_\varepsilon(\cdot,t))$ in the inequality above and thus deduce
	\begin{align}\label{est:phasefield_Besov}
	&\int_\G \int_\G \frac{\abs{\varphi_\eps(x,t)-\varphi_\eps(y,t)}^2}{d(x,y)^{2+2\alpha}} \ d\mathcal{H}^2(x) \ d\mathcal{H}^2(y) \nonumber \\ &\quad \leq \  C\norm{H(\varphi_\varepsilon(\cdot,t))}_{B_{1,1}^{2\alpha}(\G)}\leq C\norm{H(\varphi_\varepsilon(\cdot,t))}_{W^{1,1}(\G)} \leq C(T).
	\end{align}
	We refer the reader to \cite{LUN} for more details on Besov spaces while the embedding $W^{1,1}(\G) \hookrightarrow B_{1,1}^{2\alpha}(\G)$ can e.g.\ be found in \cite[Corollary 6.14]{AB}. 
	
	Observe that \eqref{est:phasefield_Besov} also holds for each localization $\alpha_i^{-1,*}(z_i\varphi_\eps)$ of $\varphi_\eps$ where $\alpha_i$ and $z_i$ are defined as in Section \ref{sec:prelim}. Because of $\int_{\mathbb{R}^2} \rho(y) dy = 1$ and \eqref{est:phasefield_Besov}, we can thus estimate
	\begin{align}\label{est:phasefield_Besov_final}
	&\int_\Gamma \left|  \varphi_\eps^\eta(x,t) - \varphi_\eps(x,t) \right|^2 \ d\mathcal{H}^2(p) \nonumber \\ 
	&=\int_\G \abs{\sum_{i\in\IS{I}} \int_{B_1(0)} \rho(y)\left[ \alpha_i^{-1,*}\left(z_i\varphi_\eps\right)\left(\alpha_i(p)-\eta y,t\right) - \alpha_i^{-1,*}(z_i\varphi_\eps)(\alpha_i(p),t)\right] \ dy}^2 \ d\mathcal{H}^2(p) \nonumber \\
	&\leq C \eta^{2+2\alpha}\int_\G \sum_{i\in\IS{I}} \int_{B_1(0)} \rho(y)\abs{y}^{2+2\alpha}\frac{\abs{ \alpha_i^{-1,*}\left(z_i\varphi_\eps\right)\left(\alpha_i(p)-\eta y,t\right) - \alpha_i^{-1,*}(z_i\varphi_\eps)(\alpha_i(p),t)}^2}{\abs{\eta y}^{2+2\alpha} } \ dy \ d\mathcal{H}^2(p) \nonumber \\
	&\leq C \eta^{2\alpha}\int_\G \sum_{i\in\IS{I}} \int_{\R^2} \frac{\abs{ \alpha_i^{-1,*}\left(z_i\varphi_\eps\right)\left(\alpha_i(p)- \tilde{y},t\right) - \alpha_i^{-1,*}(z_i\varphi_\eps)(\alpha_i(p),t)}^2}{\abs{\tilde{y}}^{2+2\alpha} } \ d\tilde{y} \ d\mathcal{H}^2(p) \nonumber \\
	&\leq C(T) \eta^{2\alpha}.
	\end{align}
	by \eqref{est:phasefield_Besov}. 
	
	For the next step, we also observe that by \eqref{eq:approx_grad_rate}
	\begin{equation*} 
	\norm{\nabla \varphi_\eps^\eta(\cdot,t)}_{L^2(\Gamma)} \leq C \eta^{-1} \norm{\varphi_\eps(\cdot,t)}_{L^2(\Gamma)} \leq C \eta^{-1} \norm{\varphi_\eps(\cdot,t)}_{L^4(\Gamma)}
	\end{equation*}
	and hence 
	\begin{equation}\label{est:grad_smooth_u}
	\norm{\nabla \varphi_\eps^\eta(\cdot,t)}_{L^2(\Gamma)} \leq \eta^{-1} C(T)
	\end{equation}
	by the estimate \eqref{eq:energy_est_lin_growth}.
	
	Again similarly as in \cite[Proof of Lemma 3.2]{ChenCH} the estimate
	\begin{align*}
	&\abs{\int_\Gamma \left( \varphi_\eps^\eta(x,t) - \varphi_\eps^\eta(x,\tau) \right) \left( \varphi_\eps(x,t) - \varphi_\eps(x,\tau) \right) \ dx} \nonumber \\
	&\quad = \big|\left\langle \left( \varphi_\eps^\eta(x,t) - \varphi_\eps^\eta(x,\tau) \right),\left( \varphi_\eps(x,t) - \varphi_\eps(x,\tau) \right) \right\rangle_{H^{-1},H^1}| \nonumber \\
	&\quad = \abs{\int_{\tau}^{t}\left\langle \SL \mu_\eps(x,s),\left( \varphi_\eps(x,t) - \varphi_\eps(x,\tau) \right) \right\rangle_{H^{-1},H^1} \ ds} \nonumber \\
	&\quad = \abs{\int_\tau^t \int_\Gamma \nabla \mu_\varepsilon(x,s) \cdot \left( \nabla \varphi_\eps^\eta(x,t) - \nabla \varphi_\eps^\eta(x,\tau) \right) \ dx \ ds} \nonumber \\
	&\quad \leq  2 \left( \int_\tau^t  \norm{\nabla \mu_\varepsilon}^2_{L^2(\G)}  \right)^{1/2} \left( t- \tau \right)^{1/2} \sup_{s \in (0,T)} \norm{\nabla \varphi_\eps^\eta(\cdot, s)}_{L^2(\Gamma)}
	\end{align*}
	holds. The energy control \eqref{eq:energy_est_lin_growth} and estimate \eqref{est:grad_smooth_u} thus yield
	\begin{align}\label{eq:control_approx_phase}
	&\int_\Gamma \left( \varphi_\eps^\eta(x,t) - \varphi_\eps^\eta(x,\tau) \right) \left( \varphi_\eps(x,t) - \varphi_\eps(x,\tau) \right) \ dx \nonumber \\
	\leq& \ C(T) \eta^{-1}(t-\tau)^{1/2} 
	\end{align}
	Choosing $\alpha=\frac{1}{4}$ and $\eta \leq (t-\tau)^{1/2},$ equation \eqref{eq:chenLemma_help} and the estimates \eqref{est:phasefield_Besov_final} and \eqref{eq:control_approx_phase} yield
	\begin{align}\label{est:chen_final}
	\int_\Gamma \left| \varphi_\eps(x,t) - \varphi_\eps(x,\tau) \right|^2 \ dx &\leq C\left( \eta^{2\alpha} + \eta^{-1}(t-\tau)^{1/2}\right)\nonumber\\ &\leq C(T)  (t-\tau)^{1/4}. 
	\end{align}
	Consequently, we deduce \eqref{eq:u_in_C}.
	
	 The estimate $\norm{H(\varphi_\eps)}_{C^{1/8}([0,T);L^1(\Gamma))} \leq C$ can be shown as in \cite[Proof of Lemma 3.2]{ChenCH}. Finally, $ \sup_{0\leq t \leq T}\norm{H(\varphi_\eps(t))}_{W^{1,1}(\Gamma)} \leq C(T)$ follows directly from Lemma \ref{l:modica_mortola_trick}.
\end{proof}

\begin{proof}[Proof of Proposition \ref{prop:conv_u}]
	The proof is based on Lemma \ref{l:estChen} and otherwise identical with the proof of \cite{ChenCH}. 
\end{proof}

\subsubsection{Weak compactness of $\lbrace \mu_\eps\rbrace$}

\begin{lemma}\label{l:weak_compactness_mu}
	There exist constants $C>0$ and $\eps_0 > 0$ such that for every $0<\eps\leq\eps_0$ and all $t \in [0,T)$ the estimate
	\begin{equation}\label{est:weak_compactness_mu}
	\norm{\mu_\eps(t)}_{H^1(\Gamma)} \leq C \left( \int_\Gamma \frac{\eps}{2}\abs{\nabla_\Gamma \varphi_\eps(t)}^2 + \frac{1}{\eps}W(\varphi_\eps(t)) \ d\mathcal{H}^2 + \norm{\nabla_\Gamma \mu_\eps(t)}_{L^2(\Gamma)} + \norm{\theta_\eps(t)}_{L^2(\Gamma)}\right)
	\end{equation}
	holds. 
\end{lemma}
\begin{proof}
	By Poincar\'{e}'s inequality and the triangle inequality, 
	\begin{align*}
	\norm{\mu_\eps(t)}_{L^2(\Gamma)} &\leq \norm{\mu_\eps(t)-\overline{\mu}_\eps(t)}_{L^2(\Gamma)} + \abs{\overline{\mu}_\eps} \\
	& \leq C \norm{\SG \mu_\eps(t)}^2_{L^2(\G)} + \abs{\overline{(\mu_\eps+\frac{1}{2}\theta_\eps)}(t)} + C \norm{\theta_\eps(t)}^2_{L^2(\Gamma)}
	\end{align*}
	where $\overline{f}$ denotes the mean value of the function $f$ over $\G.$
	We define $\omega_\eps(t):=\left(\mu_\eps+\frac{1}{2}\theta_\eps\right)(t).$ If we keep \eqref{eq:energy_est_lin_growth} in mind, it is sufficient to control the mean value of $\omega_\eps$ if we want to prove \eqref{est:weak_compactness_mu}. 
	
	Furthermore, $\omega_\eps$ solves
	\[ \omega_\eps = - \eps \SL \varphi + \eps^{-1}W'(\varphi)
	\quad \text{ on } \Gamma \times (0,T] \]
	weakly due to \eqref{eq:CH2}.
	
	The proof can thus be reduced to the proof of \cite[Lemma 3.2]{ChenCH}. As shown there, the mean value $\overline{\omega}_\eps$ fulfils for any tangential $C^1$ vector field $Y$ on $\Gamma$ 
	\begin{align}\label{eq:mean_value_frac}
	\overline{\omega}_\eps =& \frac{1}{\int_\Gamma \varphi_\eps \div_\Gamma Y \ d\mathcal{H}^2} \bigg\lbrace \left. \int_\Gamma \SG Y:\left( \left(\frac{\varepsilon \left| \nabla \varphi_\eps(x,t) \right|^2}{2} + \frac{W(\varphi_\eps(x,t))}{\varepsilon}\right) \Id \right.\right. \nonumber \\ &  - \eps \nabla \varphi_\eps(x,t) \otimes \nabla \varphi_\eps(x,t)\bigg) - \varphi_\eps Y \cdot \SG\omega_\eps - \varphi_\eps \div_\G Y \left(\omega_\varepsilon - \overline{\omega}_\eps\right) \ d\mathcal{H}^2 \bigg\rbrace. 
	\end{align}
	
	The single necessary modification here concerns the choice of $Y.$ To this end, let $\lbrace \varphi_{\eta,\eps} \rbrace_{\eta > 0} \subset C^\infty(\Gamma)$ be the family of functions given by $\varphi_{\eta,\eps} = T_\eta \varphi_\eps.$ In particular, $\varphi_{\eta,\eps}$ approximates $\varphi_\eps$ and the estimates from Lemma \ref{l:approx_on_manifolds} are fulfiled. We then define $\Psi$ to be the solution to 
	\begin{alignat*}{2}
	\Delta_\Gamma \Psi &= \varphi_{\eta,\eps} - \overline{\varphi}_{\eta,\eps}  &\quad \text{ on } \Gamma,\quad
	\int_\Gamma \Psi = 0.
	\end{alignat*}
	By Proposition \ref{prop:schauder_laplace} this solution exists and we have
	\begin{equation}\label{eq:norm_testfct}
	\norm{\Psi}_\C{2} \leq C \norm{\varphi_{\eta,\eps}}_\C{1}.
	\end{equation}
	Together with tedious calculations, this estimate and Lemma \ref{l:approx_on_manifolds} allow us to deduce that we can find a (possibly small) $\eta > 0$ and $\eps_0 >0$ such that 
	\begin{align*}
	\abs{\overline{\omega_\eps}}\leq\frac{2C\eta^{-1}\left(1+\eps^{1/2}\eta^{n/2}\right)\left( \mathcal{F}(\varphi_\eps,v_\eps) + \norm{\SG\omega_\eps}_\L{2} \right)}{\abs{\G}(1-m_0^2)},
	\end{align*}
	which proves the lemma. The omitted details can be found in the original proof by Chen or, in the special situation treated here, in \cite[Lemma 8.11]{Diss}.
\end{proof}

\subsection{Proof of the upper bound for the discrepancy measure}

This section is the beginning of the second part of the proof for the main convergence results.  The aim here is to prove the following upper bound for the positive part of the discrepancy measure $\xi^\eps(\varphi_\eps),$ which is defined as
\[ \xi^\eps(\varphi_\eps) := \left( \frac{\eps}{2} \abs{\SG \varphi_\eps}^2 - \frac{1}{\eps}W(\varphi_\eps) \right). \]
This mainly depends on the size of $\tilde{\mu}_\eps := \mu_\eps + \frac12\theta_\eps$.

\begin{proposition}\label{prop:control_discr_measure}
  There exists a positive constant $\eta_0 \in (0,1]$ and continuous, non-increasing and positive functions $M_1$ and $M_2$ defined on $(0,\eta_0]$ such that for every $\eta \in (0,\eta_0],$ every $\eps \in \left( 0,\frac{1}{M_1(\eta)}\right]$ and every $\varphi_\eps \in H^2(\G)$ and
  \begin{equation*}
    \tilde{\mu}_\eps := -\Delta_\G \varphi_\eps +\eps^{-1} W'(\varphi_\eps)\qquad \text{on }\Gamma
  \end{equation*}
we have the estimate
	\begin{align}\label{eq:control_discr_measure}
	\int_\Gamma \left( \xi^\eps(\varphi_\eps) \right)_+ \ d\mathcal{H}^2 \leq \eta \int_\Gamma \frac\eps2|\nabla\varphi|^2 +
	\eps^{-1}W(\varphi)  \ d\mathcal{H}^2  + \eps M_2(\eta)\int_\Gamma |\tilde{\mu}_\eps|^2  \ d\mathcal{H}^2 
	\end{align}
	holds.
\end{proposition}

The corresponding upper bound for a bounded domain $\Omega$ is the key part of Chen's convergence proof in \cite{ChenCH}, see Theorem 3.6 there. As stated earlier, the basic idea is to study the localized equations in each chart (and thus studying equations in the Euclidean space, as Chen did) and to prove Proposition \ref{prop:control_discr_measure} by a blow-up argument. The core of the proof lies in the following Lemma \ref{l:blow_up_chen} and its application in the blow-up argument used later in the proof of Proposition \ref{prop:control_discr_measure}.

\begin{lemma}[{\thmcite{Lemma 4.1}{ChenCH}}] \label{l:blow_up_chen}
	Assume that $\Phi \in W^{1,2}_{loc}(\R^n)$ satisfies the equation
	\begin{equation*}
	\Delta \Phi = W'(\Phi) \text{ in } \R^n.
	\end{equation*}
	Then $\Phi \in C^3(\R^n), -1 \leq \Phi \leq 1$ in $\R^n,$ and
	\begin{equation*}
	\frac{1}{2}\abs{\nabla \Phi(x)}^2 \leq W(\Phi(x)) \ \forall x \in \R^n.
	\end{equation*}
\end{lemma}
\begin{proof}
	See Lemma 4.1 in \cite{ChenCH}.
\end{proof}

The outline of the proof of Proposition \ref{prop:control_discr_measure} now is as follows: In Lemma \ref{l:chen_bulk_control} we prove an estimate which allows us to control $E_\eps (\varphi_\eps)$ away from the interface between the regions $\lbrace \varphi=1\rbrace$ and $\lbrace \varphi=-1\rbrace.$ We will then introduce rescaled coordinates on the manifold $\G$ and prove Lemma \ref{l:chen_special_case} which gives a localized version of estimate \eqref{eq:control_discr_measure} in these coordinates under the assumption that $\tilde{\mu}_\eps$ is sufficiently small. We remark that here the careful choice of a suitable atlas of $\G$ is a delicate point. The proof of this lemma will be based on Lemma \ref{l:blow_up_chen}. Finally, it will be possible to combine the local results in Lemma \ref{l:chen_bulk_control} and Lemma \ref{l:chen_special_case} to derive estimate \eqref{eq:control_discr_measure} on the entire manifold $\G.$   

Hence we start with an estimate on $E_\eps (\varphi_\eps)$ in the regions away from the interface.
\begin{lemma}\label{l:chen_bulk_control}
  There exist positive constants $C_0>0$ and $\eta_0>0$ such that for every $\eta \in (0,\eta_0],$ every $0<\eps\leq 1$ and for every
every $\varphi_\eps \in H^2(\G)$ and
  \begin{equation*}
    \tilde{\mu}_\eps := -\Delta_\G \varphi_\eps +\eps^{-1} W'(\varphi_\eps)\qquad \text{on }\Gamma
  \end{equation*}
  the estimate
	\begin{align*}
	&\int_{\lbrace p\in\G\left| \abs{\varphi_\eps}\geq 1-\eta \rbrace\right.} \frac\eps2|\nabla\varphi_\eps(p)|^2 +
	\frac{1}{\eps}W(\varphi_\eps(p)) + \frac1\eps \left(W'(\varphi_\eps(p))\right)^2 \ d\mathcal{H}^2(p) \nonumber \\
	\leq& C_0\eta \int_{\lbrace p\in\G\left| \abs{\varphi_\eps}\leq 1-\eta\rbrace\right.} \eps\abs{\SG\varphi_\eps(p)}^2 \ d\mathcal{H}^2(p) + C_0\eps\int_\G |\tilde\mu_\eps(p)|^2 \ d\mathcal{H}^2(p)   
	\end{align*}
	holds.
\end{lemma}
\begin{proof}
	The proof is given in \cite[Lemma 4.4]{ChenCH} in the case that $\Gamma$ is replaced by a bounded sufficiently smooth domain in $\R^n$ and $\varphi_\eps$ satisfies homogeneous Neumann boundary conditions. It relies on a clever testing procedure \eqref{eq:CH2}. Chen proves the desired estimate for an elliptic equation of the form of \eqref{eq:CH2} for any right-hand side in \eqref{eq:CH2}. Choosing $\tilde\mu_\eps$ as the right-hand side in \eqref{eq:CH2} one can easily prove the statement in our case by a simple adaptation of the proof of \cite[Lemma 4.4]{ChenCH}.
\end{proof}

\subsubsection{Local estimates on the discrepancy measure}

Since the manifold $\Gamma$ is compact, it is possible to cover it with a finite atlas. Thus the first step towards the proof of Proposition \ref{prop:control_discr_measure} is to prove that for functions $\varphi_\eps$ and  $\tilde\mu_\eps$ fulfilling the assumptions of the proposition a certain local version of the desired estimate on the discrepancy measure holds. 

We will first work under the assumption that $\tilde\mu_\eps$ is sufficient small before turning our attention to the cases in which $\tilde\mu_\eps$ is large. The argument requires us to carefully choose local coordinates.

As in Section \ref{sec:prelim}, we start again with normal coordinates induced by the exponential map $\exp_p$ around every point $p\in\G.$ By the compactness of $\G$, there is a real number $r>0$ such that for every $p\in\G$ the maps $\exp_p$ are diffeomorphisms from $B_r(0)$ onto the corresponding images.

Let now $R>2$ be arbitrary. Thus we can introduce the rescaled injectivity radius $\tilde r := \frac{r}{R}.$ Then the maps $\exp_p$ are still diffeomorphisms from $B_{\tilde r}(0) \subset B_{\frac{r}{2}}(0)$ onto a suitable neighborhood of $p\in\G.$ 
For $\eps \leq \frac{\tilde r}{R}=\frac{r}{R^2}$ we can choose a finite collection of points (possibly depending on the factor $R$) $\lbrace p_i \in \G \rbrace_{i=1}^{K(R)}$ such that the domains $\exp_{p_i}(B_{\eps}(0))$ cover the compact manifold. Then surely 
\[ \Big\{ \left.\exp_{p_i}\right|_{B_{\tilde r}(0)}(B_{\tilde r}(0)), \left.(\exp_{p_i})^{-1}\right|_{B_{\tilde r}(0)} \Big\}_{i=1}^{K(R)}\] 
is an atlas of $\G$ which covers $\G$ by even larger domains and has the technical advantage that after rescaling, it will be possible to cover $\G$ by images of the unit ball. 

If we denote the metric tensor by $g_{\exp}(\cdot,\cdot),$ we define $g_{\exp,ij}$ to be its entries with respect to these coordinates, i.e. $g_{\exp,ij}:=g_{\exp}(\partial_i,\partial_j)$ and let $\abs{g_{\exp}}:=\det\left( \left(g_{\exp,ij}\right)_{i,j=1}^n \right).$ Moreover, we denote the entries in the inverse $\left( \left(g_{\exp,ij}\right)_{i,j=1}^n\right)^{-1}$ by $g^{ij}_{\exp}.$

We now choose rescaled coordinates
\begin{align*}
B_{\frac{\tilde{r}}{\varepsilon}}(0) \rightarrow B_{\tilde{r}}(0), \ 
y \mapsto \varepsilon y.
\end{align*}
For $\varepsilon$ sufficiently small it is possible to choose $R>2$ such that $\varepsilon = \frac{\tilde{r}}{R}.$ We proceed by defining for all $y \in B_{\frac{\tilde r}{\varepsilon}}(0)=B_R(0)$ the functions
\begin{equation*}
F_i(y):=\varphi_\eps(\exp_{p_i}(\eps y))
\end{equation*}
and 
\begin{equation*}
M_i(y):=\eps \left(\mu_\eps(\exp_{p_i}(\eps y)) + \theta_\eps(\exp_{p_i}(\eps y))\right)
\end{equation*}
where $i = 1, \ldots ,K.$

The functions $F_i$ and $M_i$ fulfil for all $\omega \in H^1(B_R(0))$
\begin{align*}
\int_{B_{R}(0)} A(\eps y)\nabla F_i(y) \cdot \nabla \omega(y) + \sqrt{\abs{g_{\exp}(\eps y)}} &W'(F_i(y)) \omega(y) \ dy \\ &= \int_{B_{R}(0)} M_i(y) \omega(y) \sqrt{\abs{g_{\exp}(y)}} \ dy  
\end{align*}
by virtue of \eqref{eq:v} where \[A(x):=\left(a_{ij}(x)\right)_{i,j=1}^{n}:=\sqrt{\abs{g_{\exp}}}\left(g^{ij}_{\exp}\right)_{i,j=1}^{n}.\]  For $L$  defined by 
\[ Lu(y):= -\sum_{i,j} \partial_{y_i} \left(a_{ij}(\eps y)\partial_{y_j} u(\eps y) \right),\] 
$F_i$ and $M_i$ are thus weak solutions to 
\[ LF_i + \widetilde W'(F_i) = \widetilde M_i \] 
in $B_R(0)$ where $\widetilde W'(y,\varphi_\eps):= \sqrt{\abs{g_{\exp}(\eps y)}} W'(\varphi_\eps)$ and $\widetilde M_i(y) := M_i(y) \sqrt{\abs{g_{\exp}(\eps y)}}.$ For further simplification, we introduce the notation $A_\eps(y):=A(\eps y).$

Observe that the manifold $\G$ is assumed to be smooth and compact and the functions $g_{\exp,ij}$ and $g^{ij}_{\exp}$ are therefore at least locally Lipschitz. As a result, they are globally Lipschitz as well. We exploit this fact to deduce for later use the estimates
\begin{align}\label{eq:est_blow_up_A}
\norm{A_\eps(y)-\Id}_{C^0(B_R(0))}=\norm{A_\eps(y)-A(0)}_{C^0(B_R(0))} &\leq C\sup_{y \in B_R(0)}\abs{\eps y } \leq C \frac{r}{R^2} R \nonumber \\
&\leq C R^{-1}
\end{align} 
and
\begin{equation} \label{eq:est_blow_up_g}
\norm{\abs{g_{\exp}(\eps y)} - 1 }_{C^0(B_R(0))} = \norm{\abs{g_{\exp}(\eps y)} - \abs{g_{\exp}(0)}}_{C^0(B_R(0))} \leq C R^{-1}.
\end{equation}
The following lemma is then a first local estimate on the discrepancy measure for the rescaled functions $F_i$ and $M_i.$
\begin{lemma}\label{l:chen_special_case}
	For every $\eta>0$ there exist a positive constant $R(\eta)>2$ such that for every $R \geq R(\eta)$ and $F_i, M_i$ weak solutions to
	\begin{equation}\label{eq:chen_special_case}
	LF_i + \widetilde W'(F_i) = \widetilde M_i
	\end{equation}  
	in $B_R(0)$ as above with the additional assumption that 
	\begin{equation}\label{eq:general_elliptic_Op_chen_eq5}
	\norm{\widetilde M_i}_{L^2(B_R(0))}\leq C R^{-1}(\eta),
	\end{equation}
	the estimate
	\begin{align}\label{eq:second_step_chen}
	\int_{B_1} &\left( A(\eps x)\nabla F_i(x)\cdot\nabla F_i(x) - 2 \sqrt{\abs{g_{\exp}(\eps x)}} W(F_i(x)) \right)_+ \ dx \nonumber \\ 
	\leq& \eta \int_{B_2} A(\eps x)\nabla F_i(x)\cdot\nabla F_i(x) + \left[W'(F_i(x))^2 + W(F_i(x)) + \abs{M_i(x)}^2\right]\sqrt{\abs{g_{\exp}(\eps x)}} \ dx \nonumber \\ 
	&+ \int_{\lbrace x \in B_1 \left| \abs{F_i}\geq 1 - \eta\right.\rbrace} A(\eps x)\nabla F_i(x)\cdot\nabla F_i(x) \ dx 
	\end{align}
	holds.
	Moreover, $R(\eta)$ is independent of $F_i$ and $M_i.$ 
\end{lemma}
\begin{proof}
	Let $B^\eta_1$ be given by $B^\eta_1 := \lbrace x \in B_1(0) \left|\abs{F_i}\leq1-\eta\right.\rbrace.$ Since
	\begin{align*}
	\int_{B_1(0)} &\left( A_\eps(x)\nabla F_i(x)\cdot\nabla F_i(x) - 2 \widetilde W(F_i(x)) \right)_+ \ dx \nonumber \\  
	=& \int_{B^\eta_1} \left( A_\eps(x)\nabla F_i(x)\cdot\nabla F_i(x) - 2 \widetilde W(F_i(x)) \right)_+ \ dx \nonumber \\
	&+ \int_{B_1(0)\setminus B^\eta_1} \left( A_\eps(x)\nabla F_i(x)\cdot\nabla F_i(x) - 2 \widetilde W(F_i(x)) \right)_+ \ dx\nonumber \\ 
	\leq& \int_{B^\eta_1} \left( A_\eps(x)\nabla F_i(x)\cdot\nabla F_i(x) - 2 \widetilde W(F_i(x)) \right)_+ \ dx \nonumber \\
	&+ \int_{\left(B^\eta_1\right)^c} A_\eps(x)\nabla F_i(x)\cdot\nabla F_i(x) \ dx
	\end{align*}
	it is then sufficient to estimate the integral over $B^\eta_1$ in order to prove the lemma.
	We distinguish the two cases
	\[ (1) \ \abs{B^\eta_1} \leq \eta^m \text{ and } (2) \ \abs{B^\eta_1} > \eta^m\]
	where $m:=\frac{2q}{q-2}$ and $q=\frac{2n}{n-2}$ for $n>2$ and $q=7$ else.
	
	Let us first consider the case $\abs{B^\eta_1} \leq \eta^m$.
	Since $W(F_i(x))$ is non-negative for all $x \in \G$, it is enough to estimate $A\nabla F_i(x)\cdot\nabla F_i(x)$ over $B^\eta_1$. To this end observe that by the compactness of $\Gamma$ we can find an upper bound on all entries in $A$ such that
	\begin{equation}\label{eq:equiv_L2-norms}
	\int_{B^\eta_1(0)} A_\eps(x)\nabla F_i(x)\cdot\nabla F_i(x) \ dx \leq C \norm{\nabla F_i}^2_{L^2(B^\eta_1(0))}.
	\end{equation}
	We thus simplify the task at hand by proving an estimate for $\norm{\nabla F_i}_{L^2(B^\eta_1(0))}.$
	
	Using Young's inequality and the Sobolev embedding $W^{1,2}(B_1(0)) \hookrightarrow L^q(B_1(0))$ we deduce
	\begin{align*}
	\norm{\nabla F_i}_{L^2(B^\eta_1)} \leq \abs{B^\eta_1}^{\frac{q-2}{2q}}\norm{\nabla F_i}_{L^q(B^\eta_1)} \leq C \eta^{\frac{m}{m}}\norm{\nabla F_i}_{W^{1,2}(B_1(0))}. 
	\end{align*}
	A standard elliptic estimate (cf. \cite[Theorem 8.8, Theorem 8.12]{GT}) yields 
	\begin{align*}
	&\norm{\nabla F_i}^2_{W^{1,2}(B_1(0))} \leq C \left( \norm{L F_i}^2_{L^2(B_2(0)}+\norm{\nabla F_i}^2_{L^2(B_2(0))}\right) \\
	&\leq C \left( \norm{\sqrt{\abs{g_{\exp}}} M_i}^2_{L^2(B_2(0)} + \norm{\sqrt{\abs{g_{\exp}}} W'(F_i)}^2_{L^2(B_2(0)} \right. \\ &\left.\qquad \qquad \qquad \qquad \qquad \qquad \qquad \qquad +\int_{B_2(0)} A_\eps(x)\nabla F_i(x)\cdot \nabla F_i(x) \ dx \right)
	\end{align*}
	where we have used the ellipticity of $A_\eps$ and that $\sqrt{\abs{g_{\exp}}}$ is bounded from below by the compactness of $\Gamma.$
	Note that we only need the $L^2$-norm of the gradient of $F_i$ on the right hand-side since the operator $L$ does not contain terms of lower order, compare also \cite[Proof of Theorem 1, \S 6.3.1]{LCE}.
	
	Together these estimates imply
	\begin{align*}
	\norm{\nabla F_i}^2_{L^2(B^\eta_1)} \leq C\eta^2 &\left( \norm{\sqrt{\abs{g_{\exp}}} M_i}^2_{L^2(B_2(0))} \right. \\
	&\left.+ \norm{\sqrt{\abs{g_{\exp}}} W'(F_i)}^2_{L^2(B_2(0)} + \int_{B_2(0)} A_\eps(x)\nabla F_i(x)\cdot \nabla F_i(x) \ dx \right)
	\end{align*}
	and thus we infer from \eqref{eq:equiv_L2-norms}
	\begin{align*}
	\int_{B_1(0)} A_\eps \nabla F_i \cdot \nabla F_i \ dx \leq C \eta^2 \int_{B_2(0)} &\left[\abs{M_i}^2 + \left(W'(F_i)\right)^2\right] \sqrt{\abs{g_{\exp}}} +  A_\eps\nabla F_i \cdot \nabla F_i \ dx \nonumber \\ &+ \int_{\left(B^\eta_1\right)^c} A_\eps \nabla F_i \cdot \nabla F_i \ dx.
	\end{align*}
	It remains to prove the estimate in the second case, namely if $\abs{B^\eta_1} \geq \eta^m.$
	To this end, we assume that the assertion of the lemma is false and proceed by contradiction. We assume that for each $j \in \N$ there exist functions $F_i^j$ and $M_i^j$, a ball $B_j$ and let $L_j$ be the local form of the Laplace-Beltrami operator with respect to the coordinates introduced above for $R=j.$ We suppose in the following that for these functions $F_i^j$ and $M_i^j$ together with the ball $B_j$, the estimate \eqref{eq:second_step_chen} is wrong. In particular, our assumptions imply that the Matrix $A_j$ associated with the operator $L_j$ fulfils 
	\begin{equation}\label{eq:est_blow_up_A_rescaled}
	\norm{A_j( y)-\Id}_{C^0(B_j(0))} \leq C j^{-1},
	\end{equation} in accordance with \eqref{eq:est_blow_up_A}.
	
	Let now $\kappa>0$ and $\zeta$ be a smooth cut-off function with $0\leq \zeta \leq 1$ such that $\zeta \equiv 1$ on $B_{\kappa}(0)$ and $\zeta \equiv 0$ outside of $B_{2\kappa}(0).$ Moreover, let $k=\frac{2q}{q-2}.$ After multiplying with $\zeta^k F^j_i$ and integrating over $B_j,$ equation \eqref{eq:chen_special_case} reads
	\begin{align*}
	0 &= \int_{B_j} \left( A_j \nabla F^j_i \right) \cdot \nabla\left( \zeta^k F^j_i \right) + \widetilde W'(F^j_i)\zeta^kF^j_i - \widetilde M_i^j \zeta^kF^j_i \ dx
	\end{align*}
	or, equivalently,
	\begin{align}\label{eq:general_elliptic_Op_chen_step2}
	\int_{B_j} \zeta^k\left(A_j \nabla F^j_i \right)\cdot \nabla F^j_i &+ \widetilde W'(F^j_i)\zeta^k F^j_i \ dx \nonumber \\ &= \int_{B_j} - k\zeta^{k-1}F^j_i\left( A_j \nabla F^j_i\right)\cdot \nabla \zeta + \widetilde M_i^j \zeta^k F^j_i \ dx 
	\end{align}
	The summand $\int_{B_j} k\zeta^{k-1}F^j_i\left( A_j \nabla F^j_i\right)\cdot \nabla \zeta \ dx$ can be estimated with the help of Young's inequality. Choosing $\tilde{q}=\frac{q}{2}$ and $\tilde{q}'=\frac{q}{q-2}$,  we thereby obtain
	\begin{align*}
	&\int_{B_j} k\zeta^{k-1}F^j_i\left( A_j \nabla F^j_i\right)\cdot \nabla \zeta \ dx \\
	=& \int_{B_j} \left( \zeta^{k/2} (A_j)^{1/2} \nabla F^j_i \right) \cdot \left( k \zeta^{k/2 - 1}F^j_i (A_j)^{1/2}\nabla \zeta \right) \ dx \\ 
	\geq& - \frac{1}{2}\int_{B_j} \left[ \zeta^k \left( A_j\nabla F^j_i\right)\cdot\nabla F^j_i + k^2\zeta^{k-2}(F^j_i)^2\left(A_j \nabla\zeta\right)\cdot \nabla \zeta \right] \ dx\\
	\geq& -\frac{1}{2} \int_{B_j} \left[ \zeta^k\left(A_j\nabla F^j_i\right)\cdot\nabla F^j_i + \delta \left(k\zeta^{k-2}(F^j_i)^2\right)^{\tilde{q}} + C_\delta \left( k(A_j\nabla \zeta)\cdot\nabla \zeta \right)^{\tilde{q}'} \right] \ dx.
	\end{align*}
	In the same way, we can estimate the integral $\int_{B_j} \widetilde M^j_i \zeta^k F^j_i \ dx$ if we use Young's inequality to deduce
	\begin{align*}
	\int_{B_j} \widetilde M^j_i \zeta^k F^j_i \ dx &\leq  \frac{1}{2} \int_{B_j} \zeta^k (\widetilde M^j_i)^2 + \zeta^k (F^j_i)^2 \ dx \\
	&= \frac{1}{2} \int_{B_j} \zeta^k (\widetilde M^j_i)^2 + \zeta^{k-2}(F^j_i)^2\zeta^2 \ dx \\
	&\leq \frac{1}{2} \left[ \int_{B_j} (\widetilde M^j_i)^2 \ dx + \delta \int_{B_j} \zeta^k (F^j_i)^q \ dx + C_\delta \int_{B_j} \zeta^k \ dx \right],
	\end{align*}
	where we have chosen the pair $\tilde{q}$ and $\tilde{q}'$ from above as the exponents in the second application of Young's inequality. Now recall that $rW'(r) \geq c_1\abs{r}^q-c_2$ and choose $\delta = \frac{c_1}{2k^{q/2}}.$ Using the last two inequalities, \eqref{eq:general_elliptic_Op_chen_step2} becomes
	\begin{equation*}
	\int_{B_j} \zeta^k\left((A_j\nabla F^j_i)\cdot \nabla F^j_i + \abs{F^j_i}^q \right) \leq C\left(c_1,c_2,q,\norm{\widetilde M^j_i}_{L^2(B_j)},\norm{\nabla\zeta}_{L^k(B_j)}\right).
	\end{equation*}
	Since the $A_j$ are uniformly elliptic and since the sequence $\lbrace \widetilde M^j_i\rbrace_{j\in\N}\subset L^2(B_r(0)\cap B_j)$ is bounded by assumption, this estimate yields for any $\kappa > 0$
	\begin{equation*}
	\norm{F^j_i}_{W^{1,2}(B_\kappa(0)\cap B_j)} \leq C = C(\kappa).
	\end{equation*}
	Using Sobolev embeddings, we hence find that $W'(F^j_i)$ is bounded in $L^2(B_\kappa(0)\cap B_j)$ and by elliptic theory (see again \cite{GT}) we deduce for any $\kappa' < \kappa$
	\begin{equation*}
	\norm{F^j_i}_{W^{2,2}(B_{\kappa'}(0)\cap B_j)} + \norm{W'(F^j_i)}_{L^2(B_{\kappa'}(0)\cap B_j)} \leq C = C(\kappa).  
	\end{equation*}
	Since $\kappa$ was arbitrary, we can write $\kappa$ instead of $\kappa'$ in the estimate above.
	
	We are now interested in the limit behavior of the tuple $(F^j_i,M^j_i)$ as $j \rightarrow \infty.$ By the previous estimates and the general assumptions in the statement of the lemma, we can deduce the existence of a subsequence $\lbrace j_k \rbrace_{k\in\N}$ with $j_k \rightarrow \infty$ such that the following convergences hold for any $\kappa>0:$
	\begin{enumerate}[label=(\roman*)]
		\item{\label{item:general_elliptic_Op_chen2} $\widetilde M_i^{j_k} \rightarrow 0$ in $L^2(B_\kappa(0))$ by \eqref{eq:general_elliptic_Op_chen_eq5}.}
		\item{\label{item:general_elliptic_Op_chen3} $F_i^{j_k} \rightarrow F$ in $W^{1,2}(B_\kappa(0))$ and for $0<\alpha<\frac{1}{2}$ in $C^{0,\alpha}(B_\kappa(0))$ for some $F \in W^{2,2}_{loc}(\R^n)$ by the compact Sobolev embedding $W^{2,2}(B_\kappa(0)) \hookrightarrow W^{1,2}(B_\kappa(0))$ or by the compact embedding $W^{2,2}(B_\kappa(0)) \hookrightarrow C^{0,\alpha}(B_\kappa(0))$ respectively.}
		\item{\label{item:general_elliptic_Op_chen4} $W'(F_i^{j_k}) \rightarrow W'(F)$ in $L^q(B_\kappa(0))$ for $q \in [1,2).$}
		\item{\label{item:general_elliptic_Op_chen5}$W(F_i^{j_k}) \rightarrow W(F)$ in $L^1(B_\kappa(0)).$}
	\end{enumerate}
	By the dominated convergence theorem and estimate \eqref{eq:est_blow_up_g}, \ref{item:general_elliptic_Op_chen4} implies
	\[ \int_{B_\kappa(0)} \sqrt{\abs{g_{\exp}}}W'(F_i^{j_k}) \omega \ dx \rightarrow \int_{B_\kappa(0)} W'(F) \omega \ dx. \]
	At the same time, estimate \eqref{eq:est_blow_up_A_rescaled} yields
	\begin{align*}
	&\abs{\int_{B_\kappa(0)} A_{j_k}(x)\nabla F_i^{j_k}(x)\cdot\nabla\omega(x) - \nabla F(x)\cdot\nabla\omega(x) \ dx} \\
	\leq& \int_{B_\kappa(0)} \abs{A_{j_k}(x)\nabla F_i^{j_k}(x)\cdot\nabla\omega(x) - \nabla F_i^{j_k}(x)\cdot\nabla\omega(x)} + \abs{\nabla F_i^{j_k}(x)\cdot\nabla\omega(x)-\nabla F(x)\cdot\nabla\omega(x)} \ dx \\
	\leq& \norm{A_{j_k} - \Id}_{L^\infty(B_\kappa(0))} \int_{B_\kappa(0)} \nabla F_i^{j_k}(x)\cdot\nabla\omega(x) \ dx + \int_{B_\kappa(0)}\abs{\nabla F_i^{j_k}(x)-\nabla F(x)}\abs{\nabla\omega(x)} \ dx
	\end{align*}
	and thus 
	\begin{align*}
	\abs{\int_{B_\kappa(0)} A_{j_k}(x)\nabla F_i^{j_k}(x)\cdot\nabla\omega(x) - \nabla F(x)\cdot\nabla\omega(x) \ dx} \rightarrow 0
	\end{align*}
	by the convergence of $F_i^{j_k}$ in \ref{item:general_elliptic_Op_chen3}.
	
	Given that $F_i^{j_k}$ and $\widetilde M_i^{j_k}$ fulfil
	\begin{align*}
	\int_{B_\kappa(0)} A_{j_k}(x)\nabla F_i^{j_k}(x)\cdot\nabla\omega(x) \ dx + \int_{B_\kappa(0)}& \widetilde W'(F_i^{j_k}(x))\omega(x) \  dx \\ 
	&= \int_{B_\kappa(0)} M_i^{j_k}(x) \omega(x) \ dx
	\end{align*}
	for each test function $\omega \in W^{1,2}(B_\kappa(0)),$ the above convergences together with $\ref{item:general_elliptic_Op_chen2}$ imply that the limit function $F$ fulfils  
	\[ -\Delta F + W'(F) = 0 \]
	weakly in $W^{1,2}(B_\kappa(0))$ for all $\kappa > 0.$ 

	As a result, we can apply Lemma \ref{l:blow_up_chen} which gives 
	\begin{equation*}
	\lim_{k\rightarrow\infty} \int_{B_1(0)} \left( A_{j_k} \nabla F_i^{j_k}\cdot \nabla F_i^{j_k} - 2 W(F_i^{j_k})\right)_+ = \int_{B_1(0)} \left(  \abs{\nabla F}^2 - 2 W(F)\right)_+ = 0
	\end{equation*}
	Hence the left hand side in \eqref{eq:second_step_chen} is non-positive in the limit $k\rightarrow\infty.$ 
	On the other hand, we assumed 
	\[ \abs{ \left\lbrace x \in B_1(0): \abs{F_i^{j_k}} \leq 1-\eta \right\rbrace} \geq \eta^m. \]
	and since $F_i^{j_k} \rightarrow F$ a.e. this also implies
	\[ \abs{ \left\lbrace x \in B_1(0): \abs{F} \leq 1-\eta \right\rbrace} \geq \eta^m. \]
	By the convergence results in \ref{item:general_elliptic_Op_chen2}--\ref{item:general_elliptic_Op_chen5} we hence find
	\begin{align*}
	\lim_{k\rightarrow\infty} &\eta \int_{B_1(0)} A_{j_k} \nabla F_i^{j_k}\cdot \nabla F_i^{j_k} + \widetilde W(F_i^{j_k}) \ dx = \eta \int_{B_1(0)} \abs{ \nabla F}^2 + W(F) \ dx \\
	&\geq \eta \int_{\lbrace x \in B_1(0) \left| \abs{F}\leq 1-\eta\rbrace\right.} W(F) \geq \eta \eta^m \min_{s\in\left[-1+\eta,1-\eta\right]} W(s)
	\end{align*}
	and thus the right hand side of \eqref{eq:second_step_chen} is uniformly positive in $k,$ in contradiction to the assumption that the converse is true. Thus the assertion of the lemma is proved.
\end{proof}

\subsubsection{Proof of Proposition \ref{prop:control_discr_measure}}
Using Lemma \ref{l:chen_bulk_control} and \ref{l:chen_special_case} we can now proceed with the proof of Proposition \ref{prop:control_discr_measure}.

As before, we consider again normal coordinates induced by a suitable rescaling of the exponential maps $\exp_p$ around every point $p\in\G.$ 
Again we denote by $r$ the injectivity radius, which is uniform on $\G$ since $\G$ is compact. Let $\eta$ be any fixed small positive constant.  Let $R(\eta)>2$ be the constant from Lemma \ref{l:chen_special_case} such that estimate \eqref{eq:second_step_chen} holds. For $\eps$ small enough choose $R > R(\eta)$ such that $\eps = \frac{r}{R^2}.$ With the same construction as before, we obtain again an atlas for the manifold $\G$ that scales with $\eps.$ 
We denote it by
\[ \Big\{ \left.\exp^\eps_{p_i}\right|_{B_{R}(0)}(B_{R}(0)), \left.(\exp^\eps_{p_i})^{-1}\right|_{B_{R}(0)} \Big\}_{i=1}^{K(r)}.\] 

\begin{remark}\label{rem:covering}
	Note that $B_2(0) \subset B_R(0)$ and that the points $p_i$ where chosen in such a way that the rescaled atlas covers $\G$ even if one restricts the charts to $B_1(0)\subset B_R(0).$	
	Moreover, we point out that by the covering theorem \cite[Theorem 2.8.14]{HF} and Remark 2.4.8 therein for each $p \in \Gamma$ the number \[ \# \left\lbrace i \in 1,\ldots,K(r) \left| p \in \left.\exp^\eps_{p_i}\right|_{B_{R}(0)}(B_{R}(0)) \right.\right\rbrace\] is bounded by some constant $C(\G)$ which only depends on the dimension of $\G$ and is in particular independent of $\eps.$	
\end{remark}
With respect to this atlas,  the localized functions $F_i, M_i$ defined as before fulfil 
\begin{align*}
\int_{B_{R}(0)} A(\eps y)\nabla F_i(y) \cdot \nabla \omega(y) + \sqrt{\abs{g_{\exp}(\eps y)}} &W'(F_i(y)) \omega(y) \ dy \\ &= \int_{B_{R}(0)} M_i(y) \omega(y) \sqrt{\abs{g_{\exp}(y)}} \ dy  
\end{align*}
for all $\omega \in H^1(B_R(0))$ and are thus weak solutions to 
\[
  LF_i + \widetilde W'(F_i) = \widetilde M_i \qquad \text{in } B_R(0).
\]

Restricting ourselves for the moment to the set of all $i\in \mathcal{I}_1 \subset \lbrace 1\ldots K\rbrace$ given by 
\begin{equation}\label{eq:cond_A}
\mathcal{I}_1 := \left\lbrace i \in \lbrace 1\ldots K\rbrace \left| \norm{\tilde\mu_\eps}_{L^2(\exp_{p_i}(B_{\tilde r}(0))\cap\Gamma)} \leq \eps^{\frac{n}{2}-1} R^{-1} \right. \right\rbrace
\end{equation}
we can thus apply Lemma \ref{l:chen_special_case} since \eqref{eq:cond_A} implies $\norm{\widetilde M_i}_{L^2(B_R(0))} \leq R^{-1}.$

Hence we have
\begin{align*}
\int_{B_1} &\left( A\nabla F_i(x)\cdot\nabla F_i(x) - 2 \sqrt{\abs{g_{\exp}}} W(F_i(x)) \right)_+ \ dx \nonumber \\ 
\leq& \eta \int_{B_2} A\nabla F_i(x)\cdot\nabla F_i(x) + \left[W'(F_i(x))^2 + W(F_i(x)) + \abs{M_i(x)}^2\right]\sqrt{\abs{g_{\exp}}} \ dx \nonumber \\ 
&+ \int_{\lbrace x \in B_1 \left| \abs{\varphi}\geq 1 - \eta\right.\rbrace} A\nabla F_i(x)\cdot\nabla F_i(x) \ dx. 
\end{align*}
Keeping in mind that by the compactness of $\G$ both $\abs{g_{\exp}}$ and $\abs{g_{\exp}}^{-1}$ are bounded, transferring back to $\varphi_\eps$ and $\tilde\mu_\eps$ on $\G$ leads to 
\begin{align*}
&\int_{\exp^\eps_{p_i}(B_1(0))} \left(\xi^\eps(\varphi_\eps) \right)_+ \ d\mathcal{H}^2 \\ \leq& \eta C  \int_{\exp^\eps_{p_i}(B_2(0))} \left( \frac\eps2|\nabla_\G \varphi_\eps|^2 +  \eps^{-1}W(\varphi_\eps) + \frac{1}{\eps} W'(\varphi_\eps)^2 + \eps \abs{\tilde\mu_\eps}^2 \right) \ d\mathcal{H}^2 \\   &+C\int_{\exp^\eps_{p_i}(B_1(0))\cap\lbrace p\in\G \left| \abs{\varphi_\eps} \geq 1-\eta \right. \rbrace} \eps \abs{\nabla_\G \varphi_\eps}^2 \d \mathcal{H}^2.
\end{align*} 	
Taking the sum over all $i\in\mathcal{I}_1$ yields
\begin{align*}
&\sum_{i \in \mathcal{I}_1} \int_{\exp^\eps_{p_i}(B_1(0))} \left(\xi^\eps(\varphi_\eps) \right)_+ \ d\mathcal{H}^2 \\ &\leq C(\G) \eta \int_{\G} \frac\eps2|\nabla_\G \varphi_\eps|^2 +  \eps^{-1}W(\varphi_\eps) + \frac{1}{\eps} W'(\varphi_\eps)^2 + \eps \abs{\tilde\mu_\eps}^2 \ d\mathcal{H}^2 \\  &+C(\G)\int_{\lbrace p\in\G \left| \abs{\varphi_\eps}\geq 1-\eta \right. \rbrace} \eps \abs{\nabla_\G \varphi_\eps}^2 \d \mathcal{H}^2,
\end{align*}
where the constant $C(\G)$ incorporates the constant from Remark \ref{rem:covering} in order to account for the possible overlap between the regions $\exp^\eps_{p_i}(B_1(0)).$

Finally, using the estimates away from the interface in Lemma \ref{l:chen_bulk_control} provides 
\begin{align}\label{eq:chen_small_right_hand_side_final}
&\sum_{i \in \mathcal{I}_1} \int_{\exp^\eps_{p_i}(B_1(0))} \left(\xi^\eps(\varphi_\eps) \right)_+ \ d\mathcal{H}^2 \nonumber \\ &\leq C\eta \int_{\G} \frac\eps2|\nabla_\G \varphi_\eps|^2 +  \eps^{-1}W(\varphi_\eps) \ d\mathcal{H}^2 + C\eps\int_{\G}  \abs{\tilde\mu_\eps}^2 \ d\mathcal{H}^2	
\end{align}
if one takes into account that $(W'(r))^2 \leq C W(r)$ whenever $\abs{r}\leq 1.$

In order to complete the proof, let us now denote $\mathcal{I}_2 := \lbrace 1,\ldots,K\rbrace \setminus \mathcal{I}_1$ and introduce \[ G^{-1} = \left(g_{\exp}^{\eps,ij}\right)_{i,j\in\lbrace 1,2 \rbrace}.  \] We observe that by interior regularity results for elliptic equations (compare \cite[Theorem 8.8, Theorem 8.12]{GT}) we have that
\begin{align*}
\int_{B_1(0)} \nabla F_i \cdot G^{-1}\nabla F_i &\sqrt{\abs{g^{\eps}_{\exp}}} \ dx \leq C \int_{B_2(0)} \left( W'(F_i)^2 + \abs{M_i}^2+\abs{F_i}^2\right) \sqrt{\abs{g^{\eps}_{\exp}}} \ dx \\
&\leq C\abs{B_1(0)} + C \int_{B_2(0)} \left( \abs{M_i}^2 + W'(F_i)^2\chi_{\lbrace \abs{F_i}\geq 1 \rbrace} \right) \sqrt{\abs{g^{\eps}_{\exp}}} \ dx
\end{align*} 
since the $F_i$ solve 
\[  LF_i + \widetilde W'(F_i) = \widetilde M_i \]
and $W'(r)$ is bounded for $\abs{r}\leq 1.$

As before, transferring back to $\varphi_\eps$ and $\tilde\mu_\eps$ on $\G$ and the fact that $\G$ is compact allows us to deduce 
\begin{align*}
\int_{\exp^\eps_{p_i}(B_1(0))} \eps \abs{\nabla_\G \varphi_\eps}^2 \leq &C\eps \int_{\exp^\eps_{p_i}(B_1(0))} |\tilde\mu_\eps|^2 \ d\mathcal{H}^2 \\ &+ C\eps^{-1}\int_{\lbrace \exp_{p_i}(B_1(0)) \left| \abs{\varphi_\eps}\geq 1 \right. \rbrace} W'(\varphi_\eps)^2 \ d\mathcal{H}^2 + C\eps^{-1} \abs{\exp^\eps_{p_i}(B_1(0))}. 
\end{align*} 
Again, we take the sum over all $i\in\mathcal{I}_2$ and allow for a constant $C(\G)$ due to possible overlap, see Remark \ref{rem:covering}. This procedure implies
\begin{align*}
\sum_{i\in\mathcal{I}_2} \int_{\exp^\eps_{p_i}(B_1(0))} \eps \abs{\nabla_\G \varphi_\eps}^2 \leq& C(\G)\eps \int_{\G} |\tilde\mu_\eps|^2 \ d\mathcal{H}^2 \\ &+ C(\G)\eps^{-1}\int_{\lbrace p\in\G \left| \abs{\varphi_\eps}\geq 1 \right. \rbrace} W'(\varphi_\eps)^2 \ d\mathcal{H}^2+ C\eps^{-1}\sum_{i\in\mathcal{I}_2} \abs{\exp^\eps_{p_i}(B_1(0))}. 
\end{align*}
Similar to the proof of Lemma \ref{l:chen_bulk_control}, we multiply \eqref{eq:CH2} by $W'(\varphi_\eps)$ and integrate over $\Gamma$ to deduce
\begin{align*}
\int_{\Gamma} \eps W''(\varphi_\eps) \abs{\SG \varphi_\eps}^2 + \frac{1}{\eps}(W'(\varphi_\eps))^2 \ d\mathcal{H}^2 \leq \int_{\G} \frac{\eps}{2}|\tilde\mu_\eps|^2 + \frac{1}{2\eps}(W'(\varphi_\eps))^2 \ d\mathcal{H}^2. 
\end{align*}
For $\abs{s}\geq 1$ we have $W''(s) > 0$ and thus we infer
\begin{align*}
\eps^{-1}\int_{\lbrace p\in\G \left| \abs{\varphi_\eps}\geq 1 \right. \rbrace} W'(\varphi_\eps)^2 \ d\mathcal{H}^2 \leq  \int_{\G} \frac{\eps}{2}|\tilde\mu_\eps|^2 \ d\mathcal{H}^2.
\end{align*}
As a direct consequence, we obtain
\begin{align}\label{eq:chen_large_right_hand_side_step1}
\sum_{i\in\mathcal{I}_2} \int_{\exp^\eps_{p_i}(B_1(0))} \eps \abs{\nabla_\G \varphi_\eps}^2 \leq C(\G)\eps \int_{\G} |\tilde\mu_\eps|^2 \ d\mathcal{H}^2 + C\eps^{-1}\sum_{i\in\mathcal{I}_2} \abs{\exp^\eps_{p_i}(B_1(0))}. 
\end{align}
In order to derive an estimate for the constant $\sum_{i\in\mathcal{I}_2} \abs{\exp^\eps_{p_i}(B_1(0))}$ observe that for all $i\in\mathcal{I}_2$ in question
\begin{align*}
\int_{\exp^\eps_{p_i}(B_1(0))} |\tilde\mu_\eps|^2 \ d\mathcal{H}^2 \geq R^{-2}\geq \eps^{-2}R^{-2}\frac{\abs{\exp^\eps_{p_i}(B_1(0))}}{\abs{B_1}}
\end{align*}
where $B_1 \subset \G$ denotes the unit Ball in $\G.$ Remark \ref{rem:covering} ensures again that the number of overlaps of the domains $\exp^\eps_{p_i}(B_1(0))$ can be controlled independently of $\eps.$ Hence we deduce
\begin{align*}
\sum_{i\in\mathcal{I}_2} \abs{\exp^\eps_{p_i}(B_1(0))} &\leq \eps^2\abs{B_1}R^2 \sum_{i\in\mathcal{I}_2} \int_{\exp^\eps_{p_i}(B_1(0))} |\tilde\mu_\eps|^2 \ d\mathcal{H}^2 \\
&\leq \eps^2\abs{B_1}R^2C(\G)\int_\G |\tilde\mu_\eps|^2 \ d\mathcal{H}^2.
\end{align*}
Thus estimate \eqref{eq:chen_large_right_hand_side_step1} reads
\begin{align}\label{eq:chen_large_right_hand_side_final}
\sum_{i\in\mathcal{I}_2} \int_{\exp^\eps_{p_i}(B_1(0))} \eps \abs{\nabla_\G \varphi_\eps}^2 \leq C\left[1+R^{2}\right]\eps\int_{\G} |\tilde\mu_\eps|^2 \ d\mathcal{H}^2.
\end{align}
Since the family
\[ \lbrace \exp^\eps_{p_i}(B_1(0)) \rbrace_{i=1}^{K(r)}\] of domains covers $\G$ thanks to the construction of the original atlas, the estimates \eqref{eq:chen_small_right_hand_side_final} and \eqref{eq:chen_large_right_hand_side_final} yield
\begin{align*}
&\int_{\G} \left(\xi^\eps(\varphi_\eps) \right)_+ \ d\mathcal{H}^2 \\ &\leq C\eta \int_{\G} \frac\eps2|\nabla_\G \varphi_\eps|^2 +  \eps^{-1}W(\varphi_\eps) \ d\mathcal{H}^2 + C\eps M(\eta)\int_{\G} \abs{\tilde\mu_\eps}^2 \ d\mathcal{H}^2
\end{align*}
which completes the proof. Note that $M$ has to depend on $\eta$ since our choice of $R$ depends on $\eta.$

\subsection{Proof of Proposition \ref{prop:collected_conv}}
The convergences of $\varphi_\epsk, \mu_\epsk, \theta_\epsk, v_\epsk$ and $u_\epsk$ follow from Proposition \ref{prop:conv_u} for $\varphi_\epsk,$ from Lemma \ref{l:weak_compactness_mu} and the weak compactness of bounded sets in reflexive Banach spaces in the case of $\mu_\epsk$ and $\theta_\epsk$ and finally directly from the energy estimate for $u_\epsk.$

\subsection{Proof of Theorem \ref{thm:varifold_sol}}

The main part of the proof is the construction of a suitable varifold $V$ that fulfils \eqref{eq:weak_varifold_formulation_curv}. To this end, we first consider the two measures $\lambda^\epsk$ and $h^\epsk$ defined by 
\begin{align*}
\lambda^\epsk :&= \left[\epsk\frac{\abs{\SG \varphi_\epsk}^2}{2} + \frac{1}{\epsk}W(\varphi_\epsk)\right]\ d\mathcal{H}^2(p) \ dt \quad \text{and}\\
h^\epsk :&= \left[ \epsk \SG\varphi_\epsk \otimes \SG \varphi_\epsk \right] \ d\mathcal{H}^2(p) \ dt. 
\end{align*}
Observe that $\lambda^\epsk$ and $h^\epsk$ are bounded by the energy estimate. Thus we can use the compactness properties of Radon measures (see e.g. \cite[Theorem 1.59]{AFP}) to deduce the existence of measures $\lambda$ and $h$ such that
\begin{align}
\left[\epsk\frac{\abs{\SG \varphi_\epsk}^2}{2} + \frac{1}{\epsk}W(\varphi_\epsk)\right]\ & d\mathcal{H}^2(p) \ dt \rightarrow \lambda(p,t) \quad \text{and}\nonumber\\
\left[ \epsk \SG\varphi_\epsk \otimes \SG \varphi_\epsk \right] \ & d\mathcal{H}^2(p) \ dt \rightarrow h(p,t)\label{eq:convergence_tensor_measure}
\end{align}
in the sense of measures.

Moreover, for any $Y,Z \in C([0,T]\times\Gamma;T\G)$ the inequality
\begin{align*}
&\int_0^T \int_\Gamma Y\cdot( \left(\epsk\SG\varphi_\epsk\otimes\SG\varphi_\epsk\right)Z) \ d\mathcal{H}^2 \ dt \\ \leq& \int_0^T \int_\Gamma \abs{Y}\abs{Z} \left( \frac{\epsk}{2}\abs{\SG\varphi_\epsk}^2 + \frac{1}{\epsk}W(\varphi_\epsk)\right) \ d\mathcal{H}^2 \ dt \\ 
&+ \int_0^T \int_\Gamma \abs{Y}\abs{Z} \left( \frac{\epsk}{2}\abs{\SG\varphi_\epsk}^2 - \frac{1}{\epsk}W(\varphi_\epsk) \right) \ d\mathcal{H}^2 \ dt
\end{align*}
holds. 

By Proposition \ref{prop:control_discr_measure} the second term on the right hand side is non-positive in the limit $k\rightarrow\infty.$ Thus we deduce
\begin{equation}\label{eq:abs_cont_measures}
\int_0^T \int_\Gamma Y\cdot ((dh)Z) \leq \int_0^T\int_\Gamma \abs{Y}\abs{Z} d\lambda,
\end{equation}
which proves that the measure $h$ is absolutely continuous with respect to $\lambda.$ Hence the Radon-Nikodym theorem \cite[Theorem 1.28]{AFP} grants the existence of a section $\omega$ in $\mathcal{L}(T\G,T\G)$ such that 
\begin{equation}\label{eq:radon_nikodym_h}
h(p,t) = \omega d\lambda(p,t).
\end{equation} 
Next we choose an index set $\mathcal{I}$ and disjoint sets $A_l \subset \G, l \in \mathcal{I}$ such that $\G = \mathop{\dot{\bigcup}_{l \in \mathcal{I}}} A_l.$ Moreover, let $(U_l,\alpha_l)_{l \in \mathcal{I}}$ be an atlas of $\G$ such that $A_l \subset U_l$ for each $l \in \mathcal{I}.$ 

With respect to the charts $\alpha_l,$ we define local measures $h_l^\epsk$ on $\alpha_j(A_l) \subset \R^2$ by setting
\begin{align*}
h_l^\epsk := \epsk \left(g^{si}g^{rj} \frac{\partial \varphi_\epsk}{\partial x_r}\frac{\partial \varphi_\epsk}{\partial x_s} \right)_{i,j=1,2} \sqrt{\abs{g}} \ dx \ dt.
\end{align*}
Note that we use Einstein summation convention with respect to $s$ and $r$ and that the measures $h_l^\epsk$ are just $h^\epsk \mres A_l$ expressed in local coordinates.
As such, the bound on $h^\epsk$ from the energy estimate carries over to $h_l^\epsk$ and we deduce the existence of measures $h_l$ such that 
\begin{align}\label{eq:convergence_local_tensor_measure}
\epsk \left(g^{si}g^{rj} \frac{\partial \varphi_\epsk}{\partial x_r}\frac{\partial \varphi_\epsk}{\partial x_s} \right)_{i,j=1,2} \sqrt{\abs{g}} \ dx \ dt \rightarrow dh_l(x,t)
\end{align}
in the sense of measures. We also introduce the measures $h_l^{i,j}$ as the limit measures for every entry in $h^\epsk_l.$

Analogously, we define the measures $\lambda^\epsk_l$ and $\lambda_l$ as $\lambda^\epsk \mres A_l$ in local coordinates and its limit measure respectively. 

Furthermore, the arguments leading to \eqref{eq:abs_cont_measures} also imply the absolute continuity $h_l << \lambda_l.$ This implies the existence of $\lambda_l$-measurable functions $\nu_{i,j}$ such that 
\begin{equation}\label{eq:radon_nikodym_h_l}
dh^{i,j}_l(x,t) = \nu^{i,j}_ld\lambda_l(x,t).
\end{equation}
At the same time, the matrix $\left( \nu^{i,j}_l\right)_{ij}$ is symmetric and positive definite by definition and can thus be written as
\begin{equation}\label{eq:eigenvector_rep}
\left( \nu^{i,j}_l\right)_{i,j=1,2}=\sum_{k=1}^2 \tilde{c}^l_k \vec \nu_k^l \otimes \vec \nu_k^l \quad\lambda_l\text{-almost everywhere.}
\end{equation}
Here the $\lbrace \vec \nu^l_k\rbrace$ are an orthonormal basis of $\R^2$ consisting of eigenvectors. The functions $\tilde c^l_k$ fulfil $\tilde c^l_k \in [0,1]$ since equation \eqref{eq:abs_cont_measures} directly shows that the matrix $\left( \nu_{i,j}\right)_{ij}$ cannot have eigenvalues larger than $1.$ Moreover, we note for later use that
\begin{equation}\label{eq:eigenvectors_to_identity}
\sum_{k=1}^2 \vec \nu_k^l \otimes \vec \nu_k^l = \Id.
\end{equation}
Let $Y \in C^1(\G,T\G)$ be a vector field on $\G.$ To simplify the following calculations, we denote the entries of the differential $D_\G Y$ in local coordinates by $d^Y_{i,j}$ for $i,j=1,2.$

Now observe that the transformation formula infers for all $l \in \mathcal{I}$ and all $Y \in C^1(\G,T\G)$
\begin{align*}
\int_0^T \int_{A_l} D_\G Y:&\left[ \epsk \SG\varphi_\epsk \otimes \SG \varphi_\epsk \right] \ d\mathcal{H}^2(p) \ dt \\ &= \int_0^T \int_{\alpha_l(A_l)} \epsk \left(d^Y_{i,j}\right)_{i,j=1,2}:\left(g^{si}g^{rj} \frac{\partial \varphi_\epsk}{\partial x_r}\frac{\partial \varphi_\epsk}{\partial x_s} \right)_{i,j=1,2} \sqrt{\abs{g}} \ dx \ dt.
\end{align*}
For the left hand-side we have
\[ \int_0^T \int_{A_l} D_\G Y:\left[ \epsk \SG\varphi_\epsk \otimes \SG \varphi_\epsk \right] \ d\mathcal{H}^2(p) \ dt \rightarrow \int_0^T \int_{A_l} D_\G Y: \omega d\lambda \]
by \eqref{eq:convergence_tensor_measure} and \eqref{eq:radon_nikodym_h}. Furthermore, the right hand-side fulfils 
\begin{align*}
\int_0^T \int_{\alpha_l(A_l)} \epsk \left(d^Y_{i,j}\right)_{i,j=1,2}:&\left(g^{si}g^{rj} \frac{\partial \varphi_\epsk}{\partial x_r}\frac{\partial \varphi_\epsk}{\partial x_s} \right)_{i,j=1,2} \sqrt{\abs{g}} \ dx \ dt \\ 
&\quad\longrightarrow \int_0^T \int_{\alpha_l(A_l)} \left(d^Y_{i,j}\right)_{i,j=1,2}: \left(\nu^{i,j}_l\right)_{i,j=1,2} d\lambda_l
\end{align*}
by \eqref{eq:convergence_local_tensor_measure} and \eqref{eq:radon_nikodym_h_l}. Therefore we deduce 
\begin{align}\label{eq:limit_measures_local_equals_global}
\int_0^T \int_{A_l} D_\G Y:\omega d\lambda &= \int_0^T \int_{\alpha_l(A_l)} \left(d^Y_{i,j}\right)_{i,j=1,2}: \left(\nu^{i,j}_l\right)_{i,j=1,2} d\lambda_l.
\end{align}

We now define the varifold $V$ as follows. We return to the functions $\tilde{c}^l_k$ in \eqref{eq:eigenvector_rep} and define 
\begin{equation}\label{def:coeff_varifold}
c_k^l(x,t):= 1+\tilde{c}_k^l(x,t) - \sum_{m=1}^2 \tilde{c}_m^l(x,t)
\end{equation}
on $\alpha_l(A_l) \times [0,T].
$
The Radon measure $V^l$ on $[0,T]\times G(A_l)$ defined by
\begin{equation*}
dV^l_t(p,S) = \sum_{k=1}^2 \alpha_l^\ast c^l_k(p,t) d\lambda(p,t)\delta_{\alpha_l^\ast \vec \nu^l_k(p,t)}(S) 
\end{equation*}
is a varifold for almost all times $t \in [0,T]$ since by \cite[Theorem 2.28]{AFP} the measures $\lambda_l$ and $h_l^{i,j}$ can be split into a spatial and a time part, i.e. there exist measures $\lambda_l^t$ and $h_l^{i,j,t}$ such that $d\lambda = d\lambda_l^t \ dt$ and $dh_l^{i,j} = dh_l^{i,j,t} \ dt.$

Finally, we define $V$ by
\begin{equation}\label{eq:def_final_varifold}
\int_0^T \int_{G_1(\G)} \eta(p,S) \ dV_t(p,S) := \sum_{l \in \mathcal{I}} \int_0^T \int_{G_1(A_l)} \eta(p,S) \ dV^l_t(p,S)\, dt
\end{equation}
for all $\eta \in C_0(G_1(\G)).$

To conclude the proof, we show that the varifold $V$ from \eqref{eq:def_final_varifold} fulfils equation \eqref{eq:weak_varifold_formulation_curv}.
Let $Y$ be any vector field $Y\in C^1(\Gamma,T\G)$ and test \eqref{eq:CH2} with by $Y\cdot \SG \varphi_\eps$. The resulting equation is
\begin{align*}
  \int_\G (\mu + \tfrac\theta2) Y\cdot \SG \varphi_\eps \, d\mathcal{H}^2 = -\int_\G D_\G Y: \left(\left(\frac{\eps |\SG \varphi_\eps|^2}2 + \frac{W(\varphi)}\eps  \right)\operatorname{Id}- \eps \SG \varphi_\eps\otimes \SG \varphi_\eps \right) \, d\mathcal{H}^2.
\end{align*}
For all such vector fields $Y,$ the convergence results from Proposition \ref{prop:collected_conv} allow us to take the limit $k\rightarrow\infty$ in latter equation. As a result, we deduce 
\begin{align}
&-\int_0^T \int_\Gamma 2\chi_{Q_t}\div_\G((\mu + \tfrac\theta2) Y) \ d\mathcal{H}^2 \ dt = -\int_0^T \int_\Gamma d_\G Y : \left( \Id - \omega \right) d\lambda \nonumber \\
=& -\int_0^T \int_\Gamma D_\G Y : \Id \ d\lambda + \sum_{l\in\mathcal{I}} \int_0^T \int_{A_l} D_\G Y : \omega \ d\lambda  \nonumber \\
=&  -\int_0^T \int_\Gamma D_\G Y : \Id \ d\lambda + \sum_{l\in\mathcal{I}} \int_0^T \int_{\alpha_l(A_l)} \left(d^Y_{i,j}\right)_{i,j=1,2}: \left(\nu^{i,j}_l\right)_{i,j=1,2} d\lambda_l \nonumber \\
=&  -\int_0^T \int_\Gamma D_\G Y : \Id \ d\lambda + \sum_{l\in\mathcal{I}} \sum_{k=1}^2 \int_0^T \int_{\alpha_l(A_l)} \left(d^Y_{i,j}\right)_{i,j=1,2}: \left(\tilde{c}^l_k \vec \nu_k^l \otimes \vec \nu_k^l\right) d\lambda_l \label{eq:curvature_eq1}
\end{align}
from \eqref{eq:limit_measures_local_equals_global} and \eqref{eq:eigenvector_rep}.

At the same time, we have
\begin{align}
\langle \delta V_t , Y \rangle =& \int_{G_1(\Gamma)} D_\G Y(p):\left( \Id - S \otimes S \right) dV_t(p,S) \nonumber \\
=&\sum_{l \in \mathcal{I}} \int_{G_1(A_l)} D_\G Y(p):\left( \Id - S \otimes S \right) dV_t(p,S) \nonumber \\ 
=& \sum_{l \in \mathcal{I}} \sum_{k=1}^2 \int_{G_1(A_l)} D_\G Y(p):\left( \Id - S \otimes S \right) \alpha_l^\ast c_k^l(p,t) \ d\lambda(p,t)\delta_{\alpha_l^\ast \vec\nu_k^l}(S) \nonumber \\
=& \sum_{l \in \mathcal{I}} \sum_{k=1}^2 \int_{A_l} D_\G Y(p):\left( \Id - \alpha_l^\ast \vec\nu_k^l \otimes \alpha_l^\ast \vec\nu_k^l \right) \alpha_l^\ast c_k^l(p,t) \ d\lambda(p,t) \label{eq:first_variation_first_step} 
\end{align}
by the definition of the first variation of the varifold $V$ and \eqref{eq:def_final_varifold}. To simplify the presentation, we now consider the individual summands for each $l \in \mathcal{I}$ in \eqref{eq:first_variation_first_step} and split the integrals into
\[
  I^l_1 :=  \sum_{k=1}^2 \int_{A_l} \alpha_l^\ast c_k^l(p,t) D_\G Y(p):\Id  \ d\lambda(p,t)
\]
and
\[
  I^l_2 := \sum_{k=1}^2 \int_{A_l} \alpha_l^\ast c_k^l(p,t) D_\G Y(p):\left(\alpha_l^\ast \vec\nu_k^l \otimes \alpha_l^\ast \vec\nu_k^l \right)  \ d\lambda(p,t).
\]
Using \eqref{def:coeff_varifold} we calculate
\begin{align}
I^l_1 &= \int_{A_l} \left( \sum_{k=1}^2 \alpha_l^\ast c_k^l(p,t) \right) D_\G Y(p):\Id  \ d\lambda(p,t) \nonumber \\
&= \int_{A_l} \left( 2 - \sum_{k=1}^2 \alpha_l^\ast \tilde{c}_k^l \right) D_\G Y(p):\Id  \ d\lambda(p,t) \nonumber \\
&= 2 \int_{A_l} D_\G Y(p):\Id  \ d\lambda(p,t) - \int_{A_l} \left(\sum_{k=1}^2 \alpha_l^\ast \tilde{c}_k^l \right) D_\G Y(p):\Id  \ d\lambda(p,t).\label{eq:first_variation_I1}
\end{align}
Furthermore, we infer from \eqref{eq:eigenvectors_to_identity} that
\begin{align}
I_2^l &= \sum_{k=1}^2 \int_{\alpha_l(A_l)} c_k^l D_\G Y : \vec\nu_k^l \otimes \vec\nu_k^l \ d\lambda_l \nonumber \\
&= \sum_{k=1}^2 \int_{\alpha_l(A_l)} \left( 1 + \tilde{c}_k^l - \sum_{m=1}^2 \tilde{c}_m^l \right) \left(d^Y_{i,j}\right)_{i,j=1,2}:\vec\nu_k^l \otimes \vec\nu_k^l \ d\lambda_l \nonumber \\
&= \int_{\alpha_l(A_l)} \left(d^Y_{i,j}\right)_{i,j=1,2}:\Id \ d\lambda_l + \sum_{k=1}^2 \int_{\alpha_l(A_l)} \tilde{c}_k^l \left(d^Y_{i,j}\right)_{i,j=1,2}:\vec\nu_k^l \otimes \vec\nu_k^l \ d\lambda_l \nonumber \\ 
&\qquad \qquad \qquad \qquad \qquad \qquad \qquad \qquad \qquad- \int_{\alpha_l(A_l)} \left(\sum_{m=1}^2 \tilde{c}_m^l\right) \left(d^Y_{i,j}\right)_{i,j=1,2}:\Id  \ d\lambda_l \nonumber \\
&= \int_{A_l} D_\G Y(p):\Id  \ d\lambda(p,t) - \int_{A_l} \left(\sum_{k=1}^2 \alpha_l^\ast \tilde{c}_k^l \right) D_\G Y(p):\Id  \ d\lambda(p,t) \nonumber \\ 
&\qquad \qquad \qquad \qquad \qquad \qquad \qquad \qquad \qquad+ \sum_{k=1}^2  \int_{\alpha_l(A_l)} \left(d^Y_{i,j}\right)_{i,j=1,2}: \left(\tilde{c}^l_k \vec \nu_k^l \otimes \vec \nu_k^l\right) d\lambda_l. \label{eq:first_variation_I2}
\end{align}
We plug \eqref{eq:first_variation_I1} and \eqref{eq:first_variation_I2} into \eqref{eq:first_variation_first_step} and obtain
\begin{align}\label{eq:curvature_eq2}
\langle \delta V_t , Y \rangle =& \sum_{l\in\mathcal{I}} \left( I_1^l + I_2^l \right) \nonumber \\
=& \sum_{l\in\mathcal{I}} \int_{A_l} D_\G Y(p):\Id  \ d\lambda(p,t) - \sum_{k=1}^2  \int_{\alpha_l(A_l)} \left(d^Y_{i,j}\right)_{i,j=1,2}: \left(\tilde{c}^l_k \vec \nu_k^l \otimes \vec \nu_k^l\right) d\lambda_l.
\end{align}
Combining \eqref{eq:curvature_eq1} and \eqref{eq:curvature_eq2}, we have thus proved \eqref{eq:weak_varifold_formulation_curv}. Finally, the other equations  can be obtained in a straight forward manner and \eqref{eq:Energy} in the same way as in \cite[Proof of Theorem~2.1]{ChenCH}.

\end{document}